\documentclass[12pt]{amsart}    
\usepackage[colorlinks]{hyperref}
\usepackage{amsmath,amssymb,latexsym, amscd}
\usepackage{mathabx}
\usepackage{exscale, cite, eps fig, graphics}
\usepackage{mathtools}
\usepackage{enumitem}
\usepackage{lipsum}
\usepackage{graphicx}
\usepackage{subfig}
\usepackage[mathscr]{euscript}
\usepackage{array}
\usepackage{multirow}

\newcommand\MyBox[2]{
  \fbox{\lower0.75cm
    \vbox to 1.7cm{\vfil
      \hbox to 1.7cm{\hfil\parbox{1.4cm}{#1\\#2}\hfil}
      \vfil}%
  }%
}

\calclayout

\renewcommand{\geq}{\geqslant}
\renewcommand{\leq}{\leqslant}

\newcommand{\T}{\mathbb{T}}
\newcommand{\J}{\mathcal{J}}
\newcommand{\B}{\mathcal{B}}
\renewcommand{\k}{\kappa}

\newcommand{\g}{\gamma}
\renewcommand{\a}{a}
\renewcommand{\b}{\beta}

\newcommand{\A}{\mathcal{A}}

\newcommand{\R}{\mathbb{R}}
\newcommand{\Z}{\mathbb{Z}}
\newcommand{\C}{\mathbb{C}}
\newcommand{\N}{\mathbb{N}}

\newcommand{\Dx}{\partial_x}

\newcommand{\oneu}{(>0,\uparrow)}
\newcommand{\oned}{(>0,\downarrow)}
\newcommand{\minu}{(\leq0,\uparrow)}
\newcommand{\mind}{(\leq0,\downarrow)}
\newtheorem{theorem}{Theorem}[section]
\newtheorem{lemma}[theorem]{Lemma}
\newtheorem{proposition}[theorem]{Proposition}

\newtheorem{definition}[theorem]{Definition} 
\newtheorem{corollary}[theorem]{Corollary}

\newtheorem*{main-theorem}{Main Theorem}
\newtheorem*{remark*}{Remark}

\newtheorem{hypothesis}[theorem]{Hypothesis}
\numberwithin{equation}{section}

\title[Transverse instability in generalized RMKP equation]{Transverse spectral Instabilities in rotation-modified Kadomtsev-Petviashvili equation and related models}
\author[Bhavna]{Bhavna$^\ast$}
\author[Pandey]{Ashish~Kumar~Pandey$^\ast$}
\author[Semenova]{Anastassiya Semenova$^\dagger$}
\address{$^\ast$Department of Mathematics, IIIT Delhi,  110020, India\\
$^\dagger$Department of Applied Mathematics, University of Washington, Seattle, WA 98195-3925, USA}
\email{bhavnai@iiitd.ac.in, ashish.pandey@iiitd.ac.in, asemenov@uw.edu}

\date{\today}
\begin{document}

\maketitle
\begin{abstract}
The rotation-modified Kadomtsev-Petviashvili equation which is also known as the Kadomtsev–Petviashvili–Ostrovsky equation, describes the gradual wave field diffusion in the transverse direction to the direction of the propagation of the wave in a rotating
frame of reference. This equation is a generalization of the Ostrovsky equation additionally having weak transverse effects. We investigate transverse instability and stability of small periodic traveling waves of the Ostrovsky equation with respect to either periodic or square-integrable perturbations in the direction of wave propagation and periodic perturbations in the transverse direction of motion in the rotation-modified Kadomtsev-Petviashvili equation. We also study transverse stability or instability in generalized rotation-modified KP equation by taking dispersion term as general and quadratic and cubic nonlinearity. As a consequence, we obtain transverse stability or instability in two dimensional generalization of RMBO eqution, Ostrovsky-Gardner equation, Ostrovsky-fKdV equation, Ostrovsky-mKdV equation, Ostrovsky-ILW equation, Ostrovsky-Whitham etc.
 \end{abstract}
\section{Introduction}
The (2+1)-dimensional rotation-modified Kadomtsev-Petviashvili (RMKP) equation \cite{Grimshaw1985EvolutionFluid,Grimshaw1998LongOcean} is
\begin{equation}\label{e:RMKP}
    \left( u_t-\beta u_{xxx}+(u^2)_x\right)_x-\gamma u+ u_{yy}=0, ~ \gamma >0, 
\end{equation}
where, $u=u(x,y,t)$, $t\in\mathbb{R}^+$ is a temporal variable, $x,y\in\mathbb{R}$ are spatial variables. Here, $x$ represents the direction of wave propagation and $y$ represents the transverse direction to the motion of the wave. The type of dispersion is determined by the coefficient $\beta$. If the coefficient is $\beta<0$ (negative dispersion), the equation simulates gravity surface waves in a shallow water channel and internal waves in the ocean, whereas if the coefficient is $\beta>0$ (positive dispersion), the equation represents oblique magneto-acoustic waves or capillary surface waves in plasma. The parameter $\gamma>0$, which is proportional to the Coriolis force, assesses the effects of rotation. Equation \eqref{e:RMKP} simplifies to the Kadomtsev-Petviashvili (KP) equation \cite{Kadomtsev1970OnMedia} in the situation $\gamma = 0$, or without the effects of rotation.
\begin{equation}\label{e:KP}
    \left( u_t+\beta u_{xxx}+(u^2)_x\right)_x+ u_{yy}=0, 
\end{equation}and after removing $y$-term, it gets converted to the Ostrovsky equation\cite{Ostrovsky1978NonlinearOcean}
\begin{equation}\label{e:ost}
    \left( u_t+\beta u_{xxx}+(u^2)_x\right)_x-\gamma u=0.
\end{equation}
On one hand, to account for the effects of rotation, the equation \eqref{e:RMKP} can be seen as a modification of the KP equation \eqref{e:KP} while on the other hand, to account for the weak transverse effects, it can be thought of as an extension of the Ostrovsky equation \eqref{e:ost}. An interesting aspect of general class of nonlinear wave equations of type \eqref{e:RMKP}  is to consider the one-dimensional waves as solutions of these two-dimensional models. {\em Transverse stability (or instability)} in a two-dimensional model is the stability (or instability) of its one-dimensional wave with respect to the perturbations imposed which are of two-dimensions.  Our goal is to investigate the (in)stability of the $y$-independent periodic traveling waves of \eqref{e:RMKP} with respect to perturbations in two dimensions.

\subsection*{Background} The RMKP equation \eqref{e:RMKP} is extension of the well-known Korteweg-de Vries (KdV) equation in terms of dimensions and rotation factor. The KdV equation plays a noteworthy part in the explanation of nonlinear wave phenomena in weakly dispersive medium.
Numerous related equations have been discovered in the nonlinear wave theory over the years, such as the Gardner equation \cite{Krishnan2011AEquation,Xu2009AnalyticPhysics} consisting both quadratic and cubic nonlinearities (especially famous in the study of internal waves in the ocean), the modified KdV equation (mKdV) \cite{Korteweg1895Waves} describing waves in the medium with cubic nonlinearity, the Benjamin–Ono (BO) and Joseph–Kubota–Ko–Dobbs equations representing internal waves in a deep ocean \cite{Ablowitz1981SolitonsTransform}, the Kadomtsev–Petviashvili equation (KP) \cite{Kadomtsev1970OnMedia,Petviashvili2016SolitaryAtmosphere} that addressed multidimensional consequences of wave beam transverse diffraction, etc. On the same list is the Ostrovsky equation \eqref{e:ost}, which models weakly nonlinear wave patterns in the ocean while accounting for Earth rotation. In particular, wave dynamics in relaxing medium \cite{Vakhnenko1999High-frequencyMedium} are described by this equation and its reduced form (the reduced Ostrovsky equation), which omits the third derivative element. The Gardner-Ostrovsky equation, which contains both quadratic and cubic nonlinearities, can be obtained by further generalising the Ostrovsky equation \cite{Holloway1999AZone}.

In a dispersive nonlinear equation like the Ostrovsky equation, periodic traveling waves or solitary waves are of special interest. The existence of solitary wave solutions and their stability of the Ostrovsky equation and its modifications have been demonstrated in \cite{Lu2012OrbitalEquation, Liu2004StabilityEquation, Levandosky2007StabilityEquation}. In \cite{Hakkaev2017PeriodicStability, Geyer2017SpectralEquation}, periodic traveling waves of \eqref{e:ost} with general nonlinearity and generalized reduced Ostrovsky equation have been formed for small values of $\g$ and it has been demonstrated that they are spectrally stable to periodic perturbations of the same period as the wave. In \cite{Grimshaw2008Long-timeEquation}, it was found that the quasi-monochromatic wave trains of the Ostrovsky equation are modulationally unstable in a relatively “short-wave” range $k_ch<kh\ll1$ and absent in the longer wave range $kh<k_ch$, where $k$ is the wave-number, $h$ is the basin depth and $k_c=\sqrt[4]{\gamma/3|\beta|}$. This conclusion has been confirmed within more general models with cubic nonlinearity (Gardner–Ostrovsky equation \cite{Whitfield2014Rotation-inducedWaves,Whitfield2015Wave-packetEquation} and Shrira equation\cite{Nikitenkova2015ModulationalDispersions}). The authors' in \cite{Bhavna2022High-FrequencyEquation} have lately investigated that the periodic traveling waves of the Ostrovsky equation \eqref{e:ost} suffers high-frequency instability if $k>\sqrt[4]{4\gamma/\beta}$ for positive values of $\beta$ and no instability for negative values of $\beta$.

Fully two-dimensional solutions of RMKP equation \eqref{e:RMKP} still need to be thoroughly analysed. The initial steps have already been taken in this regard, namely for a very rare variant of the equation with the ``anomalous dispersion" \cite{Stepanyants2006LumpEquation,Chen2008SolitaryEquation,Esfahani2010DecayEquation,Esfahani2017StabilityRotation,Clarke2018DecayMedia,Grimshaw1990TheStudy}. In particular, as demonstrated in \cite{Esfahani2010DecayEquation}, this equation permits for fully localized solitary waves with exponential decay along the axis of transverse motion, the $y$-axis, and exponential asymptotics along the axis of motion, the $x$-axis. This is consistent with the numerical outcomes shown in \cite{Stepanyants2006LumpEquation}. 

\subsection*{Transverse stability} In a $(2+1)$-dimensional water wave equation, transverse stability (or instability) refers to the stability (or instability) of $y$-independent wave that propagates in the $x$-direction with respect to perturbations that have a nontrivial reliance on the transverse $y$-direction \cite{Haragus2011TransverseEquation,Bhavna2022TransverseEquation}. As far as the authors are aware, no conclusion has yet been made on the stability characteristics of periodic traveling waves or solitary waves of the RMKP equation \eqref{e:RMKP}. However, the transverse instability of periodic traveling waves has been studied for many similar equations, for instance, for the KP equation in \cite{Bhavna2022TransverseEquation,Hakkaev2012TransverseEquations,Johnson2010TransverseEquation,Haragus2011TransverseEquation}, for Zakharov-Kuznetsov (ZK) equation in \cite{Chen2012AEquation,Johnson2010TheEquations}. Several authors have also investigated the transverse instability of the solitary wave solutions of numerous water-wave models, see \cite{Groves2001TransverseWaves,Pego2004OnTension,Rousset2009TransverseModels,Rousset2011TransverseWater-waves}. Motivated by the importance of nonlinear waves propagation and its stability, we investigate the transverse spectral instability of the RMKP equation \eqref{e:RMKP}. We aim to study the (in)stability of the $y$-independent, that is, (1+1)-dimensional periodic traveling waves \eqref{e:expptw} of \eqref{e:RMKP} with respect to either periodic or non-periodic perturbations in the $x$-direction and always periodic in the $y$-direction. The periodic nature of the perturbations in the $y$-direction is categorized into two types: long wavelength and finite or short-wavelength.

\subsection*{One-dimensional periodic traveling waves} Since the $y$-independent form of the \eqref{e:RMKP} equation is the Ostrovsky equation \eqref{e:ost}, we are interested in small amplitude periodic traveling waves of \eqref{e:ost}. The existence of the one-parameter family of periodic traveling wave solutions of \eqref{e:ost} have been proved in \cite[Appendix~A]{Bhavna2022High-FrequencyEquation} by using Lyapunov-Schmidt Procedure, and their small amplitude exapnsion has also been calculated in \cite[Theorem~2.1]{Bhavna2022High-FrequencyEquation} which we briefly describe here. Depending on the sign of $\beta$, take into account the following wavenumbers,
\begin{enumerate}
    \item for $\beta < 0$, all $k>0$, and
    \item for $\beta > 0$, all $k>0$ but $ \left(\frac{\g}{\b n^2}\right)^{1/4}$, $2\le n\in \mathbb{N}$.
\end{enumerate}
Consequently, a one parameter family of solutions of \eqref{e:ost} exists for all such wavenumbers $k$, given by $u(x,t)=\eta(\varepsilon;k)(k(x-c(\varepsilon;k)t))$ for $\varepsilon \in \R$ and $|\varepsilon|$ sufficiently small; $\eta(\varepsilon;k)(\cdot)$ is $2\pi$-periodic, even and smooth in its argument, and $c(\varepsilon;k)$ is even in $\varepsilon$; $\eta(\varepsilon;k)$ and $c(\varepsilon;k)$ depend analytically on $\varepsilon$ and $k$. Moreover, 
\begin{equation}\label{e:expptw}
\left\{
\begin{aligned}
    \eta(\varepsilon;k)(z)=&\varepsilon\cos(z) + \varepsilon^2\eta_2\cos 2z + \varepsilon^3\eta_3\cos 3z + O(\varepsilon^4),\\
     c(\varepsilon;k)=&c_0+\varepsilon^2c_2+O(\varepsilon^4)
\end{aligned}
\right.
\end{equation}
as $\varepsilon \to 0$, 
\[
\eta_2 = \dfrac{2k^2}{3\g-12\b k^4}, \quad  \eta_3 = \dfrac{9k^2\eta_2}{8\g-72\b k^4},\quad c_0=\frac{\g}{k^2}+\b k^2,\quad c_2=\eta_2.
\]

\subsection*{Main results}
Our main findings are the following theorems, which, depending on the kind of the two-dimensional perturbations in the $x$- and $y$-directions, show the transverse stability and instability of small amplitude periodic traveling waves \eqref{e:expptw} of \eqref{e:RMKP}.

\begin{theorem}[Transverse instability]\label{t:1}
For a fixed $\g>0$ and $\b>0$, sufficiently small amplitude periodic traveling waves \eqref{e:expptw} of the RMKP equation \eqref{e:RMKP} are transversely unstable with  respect to co-periodic perturbations in the $x$-direction and, periodic with long wavelength perturbations in the $y$- direction if 
\[
k>\left|\dfrac{\g}{4\b}\right|^{1/4}.
\]
\end{theorem}

\begin{theorem}[Transverse instability]\label{t:2}
For a fixed $\g>0$ and $\b>0$, sufficiently small amplitude periodic traveling waves \eqref{e:expptw} of the RMKP equation \eqref{e:RMKP} are transversely unstable with  respect to non-periodic (localized or bounded) perturbations in the direction of propagation of the wave and, periodic with finite wavelength perturbations in the transverse direction if 
\[
k>\left|\dfrac{4\g}{\b }\right|^{1/4}.
\]
\end{theorem}

\begin{theorem}[Transverse stability]\label{t:3}
Assume that small-amplitude periodic traveling waves\eqref{e:expptw} of the Ostrovsky equation \eqref{e:ost} are spectrally stable in $L^2(\mathbb T)$ with respect to one-dimensional perturbations. Then, for $\beta\leq0$, $\g>0$, and $k>0$, periodic traveling waves \eqref{e:expptw} of the RMKP equation\eqref{e:RMKP} are transversely stable with respect to either periodic or non-periodic (localized or bounded) perturbations in the direction of wave propagation and, periodic in the transverse direction.
\end{theorem}
\begin{theorem}[Transverse stability]\label{t:4}
Assume that small-amplitude periodic traveling waves\eqref{e:expptw} of the Ostrovsky equation \eqref{e:ost} are spectrally stable in $L^2(\mathbb T)$ with respect to one-dimensional perturbations. Then, for $\beta>0$, $\g>0$, and $k>0$, periodic traveling waves \eqref{e:expptw} of the RMKP equation\eqref{e:RMKP} are transversely stable with respect to periodic perturbations in the direction of wave propagation and, periodic with finite wavelength in the transverse direction.
\end{theorem}

\subsection*{Generalized RMKP Model}
By generalizing the dispersion and nonlinearity in the RMKP equation \eqref{e:RMKP}, we put forward the generalized Rotation-Modified Kadomtsev-Petviashvili (gRMKP) equation
\begin{equation}\label{e:gRMKP}
\left( u_t+\mathcal J u_x+\alpha_1(u^2)_x+\alpha_2(u^3)_x\right)_x-\gamma u+ u_{yy}=0, ~ \gamma >0, 
\end{equation}
Here, the multiplier operator $\J$ is represented by the symbol $\jmath(k)$ as 
\begin{equation}\label{e:M}
    \widehat{\mathcal J f}(k)=\jmath (k)\widehat{f}(k),
\end{equation} 
$\alpha_1\in\{0,1\}$, and $\alpha_2\in\{-1,0\}$. This equation can also be viewed as the extension of the generalized Ostrovsky equation
\begin{equation}\label{e:gost}
    \left( u_t+\mathcal J u_x+\alpha_1(u^2)_x+\alpha_2(u^3)_x\right)_x-\gamma u=0
\end{equation}
Following are the assumptions on $\jmath(k)$.
\begin{hypothesis}\label{h:m} The multiplier symbol $\jmath(k)$ in \eqref{e:M} must meet the requirements below.
\begin{enumerate}[label=J\arabic*. , wide=0.5em,  leftmargin=*]
    \item $\jmath$ is real valued, even and without loss of generality, $\jmath(0)=1$,
    \item $A_1 k^\mathfrak{b} \leq \jmath(k) \leq A_2 k^\mathfrak{b}$ , 
$ k >> 1$ , $\mathfrak{b} \geq -1$ and for some $A_1 , A_2 > 0$,
\item $\jmath$ is strictly monotonic for $k>0$.
\end{enumerate}
\end{hypothesis}
Hypotheses J1 and J2 are necessary for the demonstration of the existence of periodic traveling waves, which we describe in the section below, while J3 is needed for the collision analysis which is done in Sections~\ref{ss:g1} and \ref{ss:g2}.

For different choices of $\jmath(k)$, $\alpha_1$, and $\alpha_2$, we obtain several models of interest.
In the similar fashion, one can define the Rotation-Modified Benjamin-Ono (RMBO) equation
 \begin{equation}\label{e:RMBO}
(u_t+\b  \mathcal{H}u_{xx} +(u^2)_x )_x-\g u=0,
\end{equation} 
by incorporating rotation effects in the Benjamin-Ono (BO) equation
 \begin{equation}\label{E:BO}
u_t+\b  \mathcal{H}u_{xx} +(u^2)_x =0.
\end{equation} 
Here, $\mathcal{H}$ is the Hilbert transform. The RMBO equation represents the long internal waves propagation
in a deep rotating fluid \cite{Grimshaw1985EvolutionFluid}. The equation \eqref{e:RMBO} can be extended to two dimensions as
 \begin{equation}\label{e:RMBOKP}
(u_t+\b  \mathcal{H}u_{xx} +(u^2)_x )_x-\g u+u_{yy}=0,
\end{equation} 
which we call Rotation-Modified Benjamin-Ono Kadomtsev-Petviashvili (RMBO-KP) equation.
We consider the following generalization of \eqref{e:RMKP} and \eqref{e:RMBOKP}
\begin{equation}\label{e:RMFKP}
(u_t+\b \Lambda^\alpha u_x +(u^2)_x )_x+u_{yy}-\g u=0,
\end{equation}
where the psuedo-differential operator $\Lambda^\alpha=(\sqrt{-\Dx^2})^\alpha$ is defined by its Fourier symbol as $\widehat{\Lambda^\alpha f}(\k)=|\k|^\alpha \widehat{f} (\k)$. Here, $1\leq\alpha\in \R$. Notice that, $\alpha=2$ and $\alpha=1$ in \eqref{e:RMFKP} gives \eqref{e:RMKP} and \eqref{e:RMBOKP} respectively. We say \eqref{e:RMFKP} as RM-fKdV-KP. In the similar fashion, one can also define the RM-Whitham-KP equation as 
\begin{equation}\label{e:RMKPW1}
(u_t+\b \mathcal{J} u_x +(u^2)_x )_x+u_{yy}-\g u=0,
\end{equation}
where 
\begin{equation}\label{e:jw1}
    \widehat{\mathcal J f}(k)=\sqrt{\dfrac{\tanh k}{k}}=:\jmath (k)\widehat{f}(k),
\end{equation}
and RMILW-KP generalizing the Intermediate Long wave (ILW) equation as 
\begin{equation}\label{e:RMKPW}
(u_t+\b \mathcal{J} u_x +(u^2)_x )_x+u_{yy}-\g u=0,
\end{equation}
where 
\begin{equation}\label{e:jw}
    \widehat{\mathcal J f}(k)=k\coth k=:\jmath (k)\widehat{f}(k),
\end{equation}
Also, RMG-KP can be defined by generalizing the Gardner equation as follows
\begin{equation}\label{e:GKP}
    (u_t-\b u_{xxx}+(u^2)_x-(u^3)_x)_x+u_{yy}-\g u=0
\end{equation}
RM-mKdV-KP equation can be defined by generalizing the mKdV equation
\begin{equation}\label{e:mKdVKP}
    (u_t-\b u_{xxx}-(u^3)_x)_x+u_{yy}-\g u=0
\end{equation}
For $\b=0$, We will call \eqref{e:RMKP} as Reduced-RMKP equation.
All these models discussed above fits into the model \eqref{e:gRMKP}.

Seeking for one-dimensional traveling wave solution of \eqref{e:gRMKP} of the form 
$u(x,y,t)=\mathscr{U}(x-c_\mathfrak{g}t)$, where $c\in\R$ is the speed of wave propagation, $\mathscr{U}$ satisfies the following
\begin{equation}\label{e:w}
-c_\mathfrak{g} \mathscr{U}^{\prime\prime}+\mathcal{J}\mathscr{U}^{\prime\prime}+\alpha_1(\mathscr{U}^2)^{\prime\prime}+\alpha_2(\mathscr{U}^3)^{\prime\prime}-\g \mathscr{U}=0
\end{equation}
We look for a periodic solution of \eqref{e:w}, that is, $\mathscr{U}$ is a $2\pi/k$-periodic function of its argument where $k>0$ is the wave number. Taking $z:=kx$, the function $\eta_\mathfrak{g}(z):=\mathscr{U}(x)$ is periodic with period $2\pi$ in $z$ and satisfies
\begin{align}\label{e:w1}
    k^2(-c_\mathfrak{g} \eta_\mathfrak{g}^{\prime\prime}+\mathcal{J}_k\eta_\mathfrak{g}^{\prime\prime}+\alpha_1(\eta_\mathfrak{g}^2)^{\prime\prime}+\alpha_2(\eta_\mathfrak{g}^3)^{\prime\prime})-\g \eta_\mathfrak{g}=0
\end{align}
Note that \eqref{e:w1} is invariant under $z\mapsto z+z_0$ and $z\mapsto -z$ and hence, we may suppose that $\eta_\mathfrak{g}$ is even. Also, observe that \eqref{e:w1} does not possess scaling invariance. We seek a non-trivial $2\pi$-periodic solution $\eta_\mathfrak{g}$ of \eqref{e:w1}. For fixed  $\g>0$, let $F:H^4(\mathbb{T})\times \R \times \R^+ \to L^2(\mathbb{T})$ be defined as
\begin{align}\label{E:F}
    F(\eta_\mathfrak{g},c_\mathfrak{g};k)=k^2(-c_\mathfrak{g} \eta_\mathfrak{g}^{\prime\prime}+\mathcal{J}_k\eta_\mathfrak{g}^{\prime\prime}+\alpha_1(\eta_\mathfrak{g}^2)^{\prime\prime}+\alpha_2(\eta_\mathfrak{g}^3)^{\prime\prime})-\g \eta_\mathfrak{g}.
\end{align}
It is well defined by a Sobolev inequality. We seek a solution $\eta_\mathfrak{g}\in H^4(\mathbb{T})$, $c \in \R$ and $k>0$ of 
\begin{equation*}\label{E:F1}
    F(\eta_\mathfrak{g},c_\mathfrak{g};k)=0.
\end{equation*}

Clearly, $F(0,c;k)=0$ for all $c_\mathfrak{g} \in \R$ and $k>0$. If non-trivial solutions of $F(\eta_\mathfrak{g},c_\mathfrak{g};k)=0$ bifurcate from $\eta_\mathfrak{g} \equiv 0$ for some $c_\mathfrak{g}$ then
\begin{equation*}\label{E:df}
    I_0 := \partial_{\eta_\mathfrak{g}}F(0,c_\mathfrak{g};k) = -c_\mathfrak{g}k^2\partial_z^2 +  k^2\mathcal{J}_k \partial_z^2 - \g
\end{equation*} from $H^4(\mathbb{T})$ to $L^2(\mathbb{T})$, is not an isomorphism. From a straightforward calculation, 
\begin{equation*}\label{E:df1}
    I_0 e^{inz} = (c_\mathfrak{g} k^2n^2 -  k^2 n^2\jmath(k) - \g) e^{inz} = 0, \quad n \in \Z
\end{equation*}
if and only if 
\[
    c_\mathfrak{g}=\jmath(k)n^2+\frac{\g}{k^2 n^2}, \quad n\in \Z.
\]
Without loss of generality, we take $n=1$ viz.
\begin{equation}\label{E:c0}
    c_{\mathfrak{g}}=\jmath(k)+\frac{\g}{k^2}.
\end{equation}
Note that wavenumbers satisfying $k^2 \left( \jmath(kn)-\jmath(k) \right) =\dfrac{\g (n^2-1)}{ n^2}$, $2\le n\in \mathbb{N}$, satisfy resonance condition
\begin{align}\label{Eq:reso}
    \frac{\g}{k^2}+\jmath(k) =\frac{\g}{k^2 n^2}+\jmath(kn)
\end{align}
of the fundamental mode and $n$th harmonic, making the kernel of $I_0$ four-dimensional. With both the kernel and the co-kernel being spanned by $e^{\pm iz}$, $I_0$ is a Fredholm operator of index zero for all other values of $k$.

 The proof of existence of a one-parameter family of non-trivial solutions of $F(\eta_\mathfrak{g},c_\mathfrak{g};k)=0$ bifurcating from $\eta_\mathfrak{g} \equiv 0$ and $c_\mathfrak{g} = c_{\mathfrak{g}_0}$ adheres to the same reasoning as the arguments in \cite[Appendix~A]{Bhavna2022High-FrequencyEquation}. We summarize the existence result for periodic traveling waves of \eqref{e:gRMKP} and their small-amplitude expansion below.
\begin{theorem}\label{T:sol}
Depending on the sign of $\jmath$, take into account the following wavenumbers,
\begin{enumerate}
    \item for $\jmath$ decreasing, all $k>0$, and
    \item for $\jmath$ increasing, all $k>0$ but those satisfying $k^2 \left( \jmath(kn)-\jmath(k) \right) =\dfrac{\g (n^2-1)}{n^2}$ ; $2\le n\in \mathbb{N}$.
\end{enumerate}
Consequently, a one parameter family of solutions of \eqref{e:w1} exists for all such wavenumbers $k$, given by $u(x,t)=\eta_\mathfrak{g}(\varepsilon;k)(k(x-c_\mathfrak{g}(\varepsilon;k)t))$ for $\varepsilon \in \R$ and $|\varepsilon|$ sufficiently small; $\eta_\mathfrak{g}(\varepsilon;k)(\cdot)$ is $2\pi$-periodic, even and smooth in its argument, and $c_\mathfrak{g}(\varepsilon;k)$ is even in $\varepsilon$; $\eta_\mathfrak{g}(\varepsilon;k)$ and $c_\mathfrak{g}(\varepsilon;k)$ depend analytically on $\varepsilon$ and $k$. Moreover, 
\begin{equation}\label{E:w_ansatz}
    \eta_\mathfrak{g}(\varepsilon;k)(z)=\varepsilon\cos(z) + \varepsilon^2\eta_{\mathfrak{g}_2}\cos 2z + \varepsilon^3\eta_{\mathfrak{g}_3}\cos 3z + O(\varepsilon^4),
\end{equation}
and
\begin{align*}
    c_\mathfrak{g}(\varepsilon;k)=c_{\mathfrak{g}_0}+\varepsilon^2c_{\mathfrak{g}_2}+O(\varepsilon^4)
\end{align*}
as $\varepsilon \to 0$, where $c_{\mathfrak{g}0}$ is in \eqref{E:c0},
\[
\eta_{\mathfrak{g}_2} = \dfrac{2\alpha_1k^2}{3\g+4 k^2(\jmath(k)-\jmath(2k))}, \quad  \eta_{\mathfrak{g}_3} = \dfrac{9\alpha_1k^2\eta_{\mathfrak{g}_2}+9/4\alpha_2 k^2}{8\g+9 k^2(\jmath(k)-\jmath(3k)},
\]
\[
c_{\mathfrak{g}_2}=\alpha_1\eta_{\mathfrak{g}_2}+(3/4)\alpha_2.
\]
\end{theorem}
 Starting now, we denote $\eta_{\mathfrak{g}}(k,\varepsilon)(z)$ and $c_\mathfrak{g}(k,\varepsilon)$ as $\eta_{\mathfrak{g}}$ and $c_\mathfrak{g}$ respectively. We study the transverse stability and instability of these waves with respect to two-dimensional perturbations of different nature. We obtain results generalizing Theorems~\ref{t:1}
-\ref{t:4}, which we list down below.

\begin{theorem}[Transverse instability]\label{t:5}
For a fixed $\g>0$, sufficiently small amplitude periodic traveling waves \eqref{E:w_ansatz} of the gRMKP equation \eqref{e:gRMKP} are transversely unstable with respect to periodic perturbations in both directions and long wavelength perturbations in the transverse direction of propagation when
\begin{enumerate}
    \item $\alpha_1=1$, $\alpha_2=0$
    \begin{enumerate}
        \item $\jmath$ is increasing, and \[|k^2(\jmath(2k)-\jmath(k))|>\dfrac{3\g}{4},\]
        \item $\jmath$ is decreasing, No instability
    \end{enumerate}
    \item $\alpha_1=0$, $\alpha_2=-1$; all $\jmath$, $\forall k>0$.
    \item $\alpha_1=1$, $\alpha_2=-1$, all $\jmath$, and
    \[\dfrac{2k^2}{3\g+4k^2(\jmath(k)-\jmath(2k))}<\dfrac{3}{4}\]

\end{enumerate}

\end{theorem}
\begin{theorem}[Transverse instability]\label{t:6}
For fixed $\g>0$ and $\jmath$ increasing, sufficiently small amplitude periodic traveling waves \eqref{E:w_ansatz} of the gRMKP equation \eqref{e:gRMKP} are transversely unstable with  respect to non-periodic (localized or bounded) perturbations in the direction of propagation of the wave and, periodic with finite wavelength perturbations in the transverse direction if 
\[k^2(\jmath(k)-\jmath(k/2))>3\g
\]
\end{theorem}
\begin{theorem}[Transverse stability]\label{t:7}
Assume that small-amplitude periodic traveling waves of the generalized Ostrovsky equation\eqref{e:gost} are spectrally stable in $L^2(\mathbb T)$ with respect to one-dimensional perturbations. Then, for $\jmath$ decreasing, $\g>0$, and $k>0$, periodic traveling waves \eqref{E:w_ansatz} of the gRMKP equation\eqref{e:gRMKP} are transversely stable with respect to
\begin{enumerate}
    \item either periodic or non-periodic perturbations in the direction of propagation and periodic perturbations in the transverse direction when $\alpha_2=0$
    \item non-periodic (localized or bounded) perturbations in the direction of propagation and periodic perturbations in the transverse direction when $\alpha_2\neq0$.
\end{enumerate}
\end{theorem}
\begin{theorem}[Transverse stability]\label{t:8}
Assume that small-amplitude periodic traveling waves of the generalized Ostrovsky equation\eqref{e:gost} are spectrally stable in $L^2(\mathbb T)$ with respect to one-dimensional perturbations. Then, for all choices of $\jmath$, $\g>0$, and $k>0$, periodic traveling waves \eqref{E:w_ansatz} of the gRMKP equation\eqref{e:gRMKP} are transversely stable with respect to periodic perturbations in the direction of propagation and, periodic with finite wavelength perturbations in the transverse direction.
\end{theorem}

Theorems~\ref{t:5}-\ref{t:8} have interesting consequences on example equations like RMBO-KP, RM-fKdV-KP, RMG-KP, RM-mKdV-KP, RM-Whitham-KP, and RMILW-KP. We list all these results in Table~\ref{tab:col1}.

\begin{table}[ht]
    \centering
    \begin{tabular}{|c|c|c|c|c|c|c|}
    \hline
       $\textbf{Equations}$  & \multicolumn{3}{c|}{\textbf{LWTP}}& \multicolumn{3}{c|}{\textbf{F/SWTP}} \\\hline
       & \multicolumn{2}{c|}{\textbf{Periodic}} & \textbf{Non-periodic}& \textbf{Periodic} & \multicolumn{2}{c|}{\textbf{Non-periodic}} \\\hline
       & $\boldsymbol{\beta>0}$ & $\boldsymbol{\beta\leq0}$ & $\boldsymbol{\beta\in\R}$ &  $\boldsymbol{\beta\in\R}$ & $\boldsymbol{\beta>0}$ & $\boldsymbol{\beta\leq0}$ \\\hline
       RMBO-KP & Unstable & Stable & Stable & Stable & Unstable & Stable \\\hline
       RM-fKdV-KP & Unstable & Stable & Stable & Stable & Unstable & Stable \\\hline
       RMG-KP & Unstable & Unstable & Stable& Stable&Unstable&Stable\\\hline
       RM-mKdV-KP & Unstable & Unstable & Stable& Stable&Unstable&Stable\\\hline
       RM-Whitham-KP & Stable & Unstable & stable & stable & Stable & Unstable \\\hline
       RMILW-KP & Unstable & Stable & Stable & Stable & Unstable & Stable \\\hline
    \end{tabular}
    \caption{Transverse stability or instability for different type of equations. Here `LWTP' stands for {\em Long wavelength transverse perturbations} and `F/SWTP' stands for {\em Finite or short wavelength transverse perturbations}. }
    \label{tab:col1}
\end{table}

In Section~\ref{sec:RMKP}, we discuss RMKP equation \eqref{e:RMKP} and prove Theorems~\ref{t:1}-\ref{t:4}. All results for RMKP equation generalizes for gRMKP equation \eqref{e:gRMKP} and analysis can be lifted from Section~\ref{sec:RMKP} as it is. We discuss this briefly in Section~\ref{sec:gRMKP} and prove Theorems~\ref{t:5}-\ref{t:8}. We discuss applications of Theorems~\ref{t:5}-\ref{t:8} in Section~\ref{sec:app} and provide a proof for results in Table~\ref{tab:col} for example equations.

\subsection*{Notations}\label{sec:notations}
The article will make use of the notations listed below. The term $L^2(\mathbb{R})$ refers to the collection of real or complex-valued, Lebesgue measurable functions $g(x)$ over $\mathbb{R}$ such that
\[
\|g\|_{L^2(\mathbb{R})}=\Big(\frac{1}{2\pi}\int_\R |g|^2~dx\Big)^{1/2}<+\infty \quad 
\]
and the term $L^2(\mathbb{T})$ refers to the space of $2\pi$-periodic, measurable, real or complex-valued functions over $\mathbb{R}$ such that
\[
\|f\|_{L^2(\mathbb{T})}=\Big(\frac{1}{2\pi}\int^{2\pi}_0 |f|^2~dx\Big)^{1/2}<+\infty. 
\]
Here, $L^2_0(\mathbb{T})$ is the space of square-integrable functions with of zero-mean, 
\begin{equation}\label{e:zerom}
    L^2_0(\T)=\left\{f\in L^2(\T)\;:\;\int_0^{2\pi}f(z)~dz=0\right\}.
\end{equation}
For $s \in \mathbb{R}$, let $H^{s}(\mathbb{R})$ consists of tempered distributions such that
$$
\|f\|_{H^{s}(\mathbb{R})}=\left(\int_{\mathbb{R}}\left(1+|t|^{2}\right)^{s}|\hat{f}(t)|^{2} d t\right)^{\frac{1}{2}}<+\infty,
$$
and
$$
H^{s}(\mathbb{T})=\left\{f \in H^{s}(\mathbb{R}): f \text { is } 2 \pi \text {-periodic }\right\}.
$$
Furthermore, the $L^2(\mathbb{T})$-inner product is defined as
\begin{equation}\label{def:i-product}
\langle f,g\rangle=\frac{1}{2\pi}\int^{2\pi}_{0} f(z)\overline{g}(z)~dz
=\sum_{n\in\mathbb{Z}} \widehat{f}_n\overline{\widehat{g}_n}.
\end{equation}

\section{RMKP equation}\label{sec:RMKP}
\subsection{Constrution of the spectral problem}
In a traveling frame of reference moving with velocity $c\in\R$, with coordinates $(x,t)\to (k(x-ct),t)$, \eqref{e:RMKP} becomes 
\begin{equation}\label{e:gkplin1}
k(\eta_t - ck \eta_z -\b k^3 \eta_{zzz}+k(\eta^2)_z)_z+\eta_{yy}-\g \eta=0,
\end{equation}where $\eta$ is the one-dimensional periodic traveling wave solution of \eqref{e:RMKP}, given in \eqref{e:expptw}.
We consider perturbations to $\eta$ of the form $\eta+\epsilon\eta^\ast+O(\epsilon^2)$ with $0<|\epsilon|\ll 1$, we arrive at
\begin{equation}\label{e:gkplin}
k(\eta^\ast_t - ck \eta^\ast_z -\b k^3 \eta^\ast_{zzz}+2k(\eta \eta^\ast)_z)_z+\eta^\ast_{yy}-\g \eta^\ast=0,
\end{equation}
We are looking for a solution of the form
\begin{equation}\label{e:per}
\eta^\ast(z,y,t)=e^{\frac{\lambda}{k} t+i\rho y}\varphi(z),\quad \lambda\in \mathbb{C},~ \rho\in\R 
\end{equation}to arrive at
\begin{equation}\label{e:opt}
\mathcal G_\varepsilon(\lambda,\rho) \varphi:=(\lambda\partial_z- k^2\partial_z^2(c+\b k^2\partial_z^2-2\eta)-\rho^2-\g)\varphi=0
\end{equation}
The one-dimensional periodic traveling wave solution $\eta$ of \eqref{e:RMKP} is transversely spectrally unstable if there exist a $\rho\neq0$ such that there is a $\lambda$ satisfying \eqref{e:opt} with $\Re(\lambda)>0$, then there exists a
perturbation \eqref{e:per} that grows exponentially in time $(t>0)$, otherwise $\eta$ is deemed to be transversely spectrally stable.
The type of perturbations is governed on the functional space in which $\varphi$ lies. We will work with two type of functional spaces - space of co-periodic functions, and space of square-integrable functions over the whole real line. Depending on the space, there will be two type of perturbations - periodic perturbations, and localized or bounded perturbations.
\begin{definition}(Transverse (in)stability)
Assuming that $2\pi/k$-periodic traveling wave solution $u(x,y,t)=\eta(k(x-ct))$ of \eqref{e:RMKP} is a stable solution of the one-dimensional equation \eqref{e:ost} where $\eta$, $c$ are as in \eqref{e:expptw}, we say that the periodic wave $\eta$ in \eqref{e:expptw} is transversely spectrally stable with respect to two-dimensional periodic perturbations (resp. non-periodic (localized or bounded perturbations)) if the RMKP operator $\mathcal G_\varepsilon(\lambda,\rho)$ acting in $L^2(\T)$ (resp.  $L^2(\R)$ or $C_b(\R)$) is invertible, for any $\lambda\in\C$, $\Re(\lambda)>0$ and any $\rho\neq 0$, otherwise it is deemed transversely spectrally unstable. 
\end{definition}

\subsubsection*{\underline{Periodic perturbations}}\label{s:per}In case of periodic perturbations, we investigate the invertibility of the operator $\mathcal G_\varepsilon(\lambda,\rho)$ acting in $L^2(\T)$ with domain $H^4(\T)$, for $\lambda\in\C$ with $\Re(\lambda)>0$, $0\neq\rho\in\R$. The invertibility problem
\[
\mathcal G_{\varepsilon}(\lambda,\rho)\varphi=0,\quad \varphi\in L^2(\T)
\]
is transformed into a spectral problem which requires invertibility of $\partial_z$. Since $\partial_z$ is not invertible in $L^2(\mathbb{T})$, we restrict the problem to mean-zero subspace $L^2_0(\mathbb{T})$, defined in \eqref{e:zerom}, of $L^2(\mathbb{T})$. The study of invertibility of the operator $\mathcal{G}_\varepsilon(\lambda,\rho)$ in $L^2(\T)$ is equivalent to the study of invertibility of the operator $\mathcal{G}_\varepsilon(\lambda,\rho)$ in $L_0^2(\T)$, as shown in the following lemma.
\begin{lemma}\label{l:1}
$\mathcal{G}_\varepsilon(\lambda,\rho)$ acting in $L^2(\T)$ with domain $H^4(\T)$ is not invertible if and only if $\lambda$ belongs to $L_0^2(\mathbb{T})$-spectrum of the operator $\mathcal Q_{\varepsilon}(\rho)$, where 
\begin{equation}\label{e:op5}
	\mathcal Q_{\varepsilon}(\rho):= k^2\partial_z(c+\beta k^2\partial_z^2-2\eta)+(\g+\rho^2)\partial_z^{-1}.\\
\end{equation}
\end{lemma}
\begin{proof}
Observe that if $\varphi\in L_0^2(\mathbb{T})$, then $\mathcal{G}_\varepsilon(\lambda,\rho)\varphi\in L_0^2(\mathbb{T})$ implies that the subspace $L_0^2(\mathbb{T})\subset L^2(\T)$ is $\mathcal{G}_\varepsilon(\lambda,\rho)$- invariant. Additionally, the spectrum is composed of discrete eigenvalues with finite multiplicity because the operator $\mathcal{G}_\varepsilon(\lambda,\rho)$ acting on $L^2(\mathbb{T})$ has a compact resolvent. Therefore, $\mathcal{G}_\varepsilon(\lambda,\rho)$ is not invertible if and only if zero is an eigenvalue of $\mathcal{G}_\varepsilon(\lambda,\rho)$, that is, if and only if there exists a non-zero $\psi\in H^4(\T)$ such that $\mathcal{G}_\varepsilon(\lambda,\rho)\psi=0$. Since $\gamma\neq0$, we conclude that $\psi\in L^2_0(\T)$. Therefore, zero is an eigenvalue of $\mathcal{G}_\varepsilon(\lambda,\rho)$ acting in $L^2(\T)$ if and only if zero is an eigenvalue of the restriction of $\mathcal{G}_\varepsilon(\lambda,\rho)$ to $L^2_0(\T)$. This implies that $\mathcal{G}_\varepsilon(\lambda,\rho)$ acting in $L^2(\T)$ with domain $H^4(\T)$ is not invertible if and only if its restriction to the subspace $L^2_0(\T)$ is not invertible. Moreover, for a $\varphi\in L_0^2(\mathbb{T})$, $\mathcal G_{\varepsilon}(\lambda,\rho)\varphi=0$ if and only if $\mathcal Q_{\varepsilon}(\rho)\varphi=\lambda \varphi$, where
\begin{equation}\label{e:op}
	\mathcal Q_{\varepsilon}(\rho):= k^2\partial_z(c+\beta k^2\partial_z^2-2\eta)+(\g+\rho^2)\partial_z^{-1}.\\
\end{equation}
Therefore, the operator $\mathcal G_{\varepsilon}(\lambda,\rho)$ is not invertible in $L^2_0(\mathbb{T})$ for some $\lambda\in \C$ if and only if  $\lambda\in\operatorname{spec}_{L^2_0(\mathbb{T})}(\mathcal Q_{\varepsilon}(\rho))$, that is, $L_0^2(\mathbb{T})$-spectrum of the operator $\mathcal Q_{\varepsilon}(\rho)$. Hence the lemma.
\end{proof}
 We arrive at pseudo-differential spectral problem
\begin{equation}\label{e:sp1}
  \mathcal Q_{\varepsilon}(\rho)\varphi=\lambda\varphi,
\end{equation}where $\varphi\in L_0^2(\T)$. With respect to periodic perturbations, we will study the spectrum of the operator $\mathcal Q_{\varepsilon}(\rho)$ acting in  $L^2_0(\T)$ with domain $H^{3}(\T)\cap L^2_0(\T)$. 
\begin{proposition}\label{p:1}
The operator $\mathcal Q_{\varepsilon}(\rho)$ possess following properties.
\begin{enumerate}
    \item The operator $\mathcal Q_{\varepsilon}(\rho)$ commutes with the reflection through the real axis.
    \item The operator $\mathcal Q_{\varepsilon}(\rho)$ anti-commutes with the reflection through the origin and the imaginary axis.
    \item The spectrum of $\mathcal Q_{\varepsilon}(\rho)$ is symmetric with respect to the reflections through origin, real axis, and imaginary axis.
\end{enumerate}
\end{proposition} 
\begin{proof}
We consider $\mathcal{R}_r$, $\mathcal{R}_i$ and $\mathcal{R}_o$ to be the reflections through the real axis, imaginary axis, and the origin, respectively, defined as follows
\begin{equation}\label{e:ref}
 \mathcal{R}_r\psi(z)=\overline{\psi(z)},\quad   \mathcal{R}_i\psi(z)=\overline{\psi(-z)}\quad\text{and}\quad \mathcal{R}_o\psi(z)=\psi(-z)
\end{equation}
Assume $\lambda$ is an eigenvalue of $\mathcal Q_{\varepsilon}(\rho)$ with an associated eigenvector $\varphi$, then we have
\begin{equation}\label{e:eig}
    \mathcal Q_{\varepsilon}(\rho)\varphi=\lambda\varphi
\end{equation}
Observe that
\begin{equation}\label{e:01}
    (\mathcal Q_{\varepsilon}(\rho)\mathcal{R}_r\psi)(z)=\mathcal Q_{\varepsilon}(\rho)(\mathcal{R}_r\psi(z))=\mathcal Q_{\varepsilon}(\rho)\overline{\psi(z)}=(\overline{\mathcal Q_{\varepsilon}(\rho)\psi})(z)=(\mathcal{R}_r\mathcal Q_{\varepsilon}(\rho)\psi)(z)
\end{equation}
\begin{equation}\label{e:2}
(\mathcal Q_{\varepsilon}(\rho)\mathcal{R}_i\psi)(z)=\mathcal Q_{\varepsilon}(\rho)(\mathcal{R}_i\psi(z))=\mathcal Q_{\varepsilon}(\rho)\overline{\psi(-z)}=-(\overline{\mathcal Q_{\varepsilon}(\rho)\psi})(-z)=-(\mathcal{R}_i\mathcal Q_{\varepsilon}(\rho)\psi)(z),
\end{equation}
\begin{equation}\label{e:3}
(\mathcal Q_{\varepsilon}(\rho)\mathcal{R}_o\psi)(z)=\mathcal Q_{\varepsilon}(\rho)(\mathcal{R}_o\psi(z))=\mathcal Q_{\varepsilon}(\rho)\psi(-z)=-(\mathcal Q_{\varepsilon}(\rho)\psi)(-z)=-(\mathcal{R}_o\mathcal Q_{\varepsilon}(\rho)\psi)(z),
\end{equation}
From \eqref{e:01}, $\mathcal Q_{\varepsilon}(\rho)$ commutes with $\mathcal{R}_r$. From \eqref{e:2} and \eqref{e:3}, we conclude that $\mathcal Q_{\varepsilon}(\rho)$ anti-commutes with $\mathcal{R}_i$ and $\mathcal{R}_o$. Using \eqref{e:eig}, we arrive at
\[\mathcal Q_{\varepsilon}(\rho)\mathcal{R}_r\varphi=\mathcal{R}_r\mathcal Q_{\varepsilon}(\rho)\varphi=\overline{\lambda}\mathcal{R}_r\varphi
\]
\[\mathcal Q_{\varepsilon}(\rho)\mathcal{R}_i\varphi=-\mathcal{R}_i\mathcal Q_{\varepsilon}(\rho)\varphi=-\overline{\lambda}\mathcal{R}_i\varphi
\]
\[\mathcal Q_{\varepsilon}(\rho)\mathcal{R}_o\varphi=-\mathcal{R}_o\mathcal Q_{\varepsilon}(\rho)\varphi=-\lambda\mathcal{R}_o\varphi
\]
We conclude from here that if $\lambda$ is an eigenvalue of $\mathcal Q_{\varepsilon}(\rho)$ with associated eigenvector $\varphi$, then $\overline{\lambda}$, $-\overline{\lambda}$ and $-\lambda$ are also eigenvalues of $\mathcal Q_{\varepsilon}(\rho)$ with associated eigenvectors $\mathcal{R}_r\varphi$, $\mathcal{R}_i\varphi$ and $\mathcal{R}_o\varphi$, respectively. Therefore, the spectrum of $\mathcal Q_{\varepsilon}(\rho)$ is symmetric with respect to the reflections through the origin, real axis and imaginary axis.
\end{proof}
\subsubsection*{\underline{Localized or bounded perturbations}} We check the invertibility of the operator $\mathcal G_\varepsilon(\lambda,\rho)$ acting in $L^2(\R)$ or $C_b(\R)$ with domain $H^{4}(\R)$ or   $C_b^{4}(\R)$, for $\lambda\in\C$ with $\Re(\lambda)>0$, $0\neq\rho\in\R$.
In $L^2(\R)$ or $C_b(\R)$, the operator $\mathcal G_\varepsilon(\lambda, \rho)$ has continuous spectrum. Since the coefficients of the operator $\mathcal G_\varepsilon(\lambda, \rho)$ are $2\pi$-periodic, we can use the Floquet theory such that all solutions of \eqref{e:opt} in $L^2(\R)$ or $C_b(\R)$ are of the form $\varphi(z)=e^{i\xi z}\Tilde{\varphi}(z)$ where $\xi\in\left(-1/2,1/2\right]$ is the Floquet exponent and $\Tilde{\varphi}$ is a $2\pi$-periodic function, see \cite{Haragus2008STABILITYEQUATION} for a comparable circumstance. By following same arguments as in the proof of \cite[Proposition A.1]{Haragus2008STABILITYEQUATION}, we can deduce that the study of the invertibility of $\mathcal G_\varepsilon(\lambda,\rho)$ in $L^2(\R)$ or $C_b(\R)$ is equivalent to the invertibility of the linear operators $\mathcal G_{\varepsilon,\xi}(\lambda,\rho)$ in $L^2(\mathbb{T})$ with domain $H^{4}(\mathbb{T})$, for all $\xi\in\left(-1/2,1/2\right]$, where
\begin{align*}\label{E:bloch}
\mathcal G_{\varepsilon,\xi}(\lambda,\rho) := \lambda(\partial_z+i\xi)- k^2(\partial_z+i\xi)^2(c+\b k^2(\partial_z+i\xi)^2-2\eta)-\rho^2-\g.
\end{align*}
We will confine ourselves to the case $\xi\neq0$ because $\xi=0$ refers to periodic perturbations. The $L^2(\mathbb{T})$-spectra of operator $\mathcal G_{\varepsilon,\xi}(\lambda,\rho)$ consist of eigenvalues of finite multiplicity. Therefore, $\mathcal G_{\varepsilon,\xi}(\lambda,\rho)$ is not invertible in $L^2(\mathbb{T})$ if and only if zero is an eigenvalue of $\mathcal G_{\varepsilon,\xi}(\lambda,\rho)$. For a $\Psi\in L^2(\mathbb{T})$, $\mathcal G_{\varepsilon,\xi}(\lambda,\rho)\Psi=0$ if and only if $\mathcal Q_{\varepsilon}(\rho,\xi)\Psi=\lambda \Psi$, where
\begin{equation}\label{e:op7}
	\mathcal{Q}_\varepsilon(\rho,\xi):=k^2(\partial_z+i\xi)(c+\beta k^2(\partial_z+i\xi)^2-2\eta)+(\g+\rho^2)(\partial_z+i\xi)^{-1}.
\end{equation}
Therefore, the operator $\mathcal G_{\varepsilon,\xi}(\lambda,\rho)$ is not  invertible in $L^2(\mathbb{T})$ for some $\lambda\in \C$ and $\xi\neq 0$ if and only if  $\lambda\in\operatorname{spec}_{L^2(\mathbb{T})}(\mathcal Q_{\varepsilon}(\rho,\xi))$, $L^2(\mathbb{T})$-spectrum of the operator.
This is straight forward to observe that $\operatorname{spec}_{L^2(\mathbb{T})}(\mathcal Q_{\varepsilon}(\rho,\xi))$ is not symmetric with respect to the reflections through the real axis and the origin, rather it exhibit following properties.
\begin{proposition}\label{p:02}
The operator $\mathcal Q_{\varepsilon}(\rho,\xi)$ possesses following properties.
\begin{enumerate}
    \item The operator $\mathcal Q_{\varepsilon}(\rho,-\xi)$ commutes with the reflection through the real axis.
    \item The operator $\mathcal Q_{\varepsilon}(\rho,\xi)$ anti-commutes with the reflection through the imaginary axis.
    \item The operator $\mathcal Q_{\varepsilon}(\rho,-\xi)$ anti-commutes with the reflection through the origin.
    \item The spectrum of $\mathcal Q_{\varepsilon}(\rho,\xi)$ is symmetric with respect to the reflections through the imaginary axis.
    \item The spectrum of $\mathcal Q_{\varepsilon}(\rho,-\xi)$ is symmetric with respect to the reflections through the real axis and the origin.
\end{enumerate}
\end{proposition}
\begin{proof}
Proof is similar to the Proposition \ref{p:1}. 
\end{proof}
As a consequence of the properties mentioned in Proposition \ref{p:02}, it is sufficient to take $\xi\in(0,1/2]$. Therefore, we arrive at pseudo-differential spectral problem
\begin{equation}\label{e:non}
    \mathcal Q_{\varepsilon}(\rho,\xi)\varphi=\lambda\varphi;\quad\varphi\in L^2(\T)\quad\text{and}\quad\xi\in(0,1/2]
\end{equation}
In case of non-periodic perturbations, we need to investigate if there exist any $\lambda\in\operatorname{spec}_{L^2(\T)}(\mathcal Q_{\varepsilon}(\rho,\xi))$ with $\Re(\lambda)>0$ for some $\rho\neq0$ and $\xi\in(0,1/2]$.
\subsection{Characterization of the potentially unstable spectrum}
\subsubsection{\underline{Periodic perturbations}}
Here, we check if there exist any $\lambda\in\operatorname{spec}_{L_0^2(\T)}(\mathcal Q_{\varepsilon}(\rho))$ with $\Re(\lambda)>0$ for some $\rho\neq0$.
As discussed in Section~\ref{s:per}, $\operatorname{spec}_{L_0^2(\T)}(\mathcal Q_{\varepsilon}(\rho))$ inherits quadrafold symmetry that means an eigenvalue of $\mathcal Q_{\varepsilon}(\rho)$ away from the imaginary axis guarantees an eigenvalue with positive real part. The spectral analysis for the operator $\mathcal Q_\varepsilon(\rho)$ is based on perturbation
arguments in which we consider $\mathcal Q_\varepsilon(\rho)$ as a perturbation of the operator $\mathcal Q_0(\rho)$ for $|\varepsilon|$ sufficiently small. More precisely, for any $d>0$, there exist $\delta>0$ such that for any $\varepsilon$ with $\|\varepsilon\|\leq\delta$, the spectrum of $\operatorname{spec}_{L_0^2(\T)}(\mathcal Q_{\varepsilon}(\rho))$ satisfies \[\operatorname{spec}_{L_0^2(\T)}(\mathcal Q_{\varepsilon}(\rho))\subset\{\lambda\in\C; \operatorname{dist(\lambda,\operatorname{spec}_{L_0^2(\T)}(\mathcal Q_{0}(\rho)))}<d\}\]
To find the spectrum of $\mathcal{Q}_\varepsilon(\rho)$, we require to find the spectrum of $\mathcal{Q}_0(\rho)$. Since $\operatorname{spec}_{L_0^2(\T)}(\mathcal Q_{0}(\rho))$ is a differential operator with periodic coefficients, a straightforward Fourier analysis allows to compute its spectrum explicitly as
\begin{align}\label{E:spec}
    \mathcal Q_{0}(\rho)e^{inz} = i\Omega_{n,\rho}e^{inz}\quad \text{for all}\quad n \in \mathbb{Z}\setminus \{0\}.
\end{align}
where \begin{align}\label{E:oomega}
    \Omega_{n,\rho} = \g\left(n-\dfrac{1}{n}\right)+\b k^4(n-n^3)-\dfrac{\rho^2}{n}
\end{align}
Therefore, the $L_0^2(\T)$-spectrum of $\mathcal Q_0(\rho)$ is given by
\begin{equation}\label{e:spec}
    \operatorname{spec}_{L^2_0(\mathbb{T})}(\mathcal Q_0(\rho))=\{i\Omega_{n,\rho}; n \in \Z^\ast \},
\end{equation}
which implies $\operatorname{spec}_{L^2_0(\mathbb{T})}(\mathcal Q_0(\gamma))$ contains purely imaginary eigenvalues with finite algebraic multiplicity, which has to be like this since $\varepsilon=0$ represents the zero solution, which is trivially stable. The eigenvalues in \eqref{e:spec} shift around and may depart from the imaginary axis when $|\varepsilon|$ rises, resulting in spectral instability. Since the spectrum around the real and imaginary axes is symmetric, it follows that for sufficiently small values of $|\varepsilon|$, eigenvalues of $\mathcal Q_\varepsilon(\rho)$ must bifurcate in pairs as a result of collisions of eigenvalues of $\mathcal Q_0(\rho)$ on the imaginary axis.

Let $n\neq m\in \Z^\ast$, a pair of eigenvalues $i\Omega_{n,\rho}$ and $i\Omega_{m,\rho}$ of $\mathcal Q_0(\rho)$ collide for some $\rho=\rho_c$ when 
\begin{align}\label{e:coll}
    \Omega_{n,\rho_c}=\Omega_{m,\rho_c}.
\end{align}
Without any loss of generality, consider $n<m$ and $m=n+\Theta$ with $\Theta\in\N$ in collision condition \eqref{e:coll}, we arrive at
\begin{equation}\label{e:ccc1}
    \rho^2=-\g \mathscr{J}(n,\Theta)+\beta k^4 \mathscr{U}(n,\Theta)
\end{equation}
where,
\[\mathscr{J}(n,\Theta)=n(n+\Theta)+1\quad\text{and}\quad \mathscr{U}(n,\Theta)=3n^2(n+\Theta)^2+n(n+\Theta)(\Theta^2-1).\]
Analyzing signs of $\mathscr{J}(n,\Theta)$ and $\mathscr{U}(n,\Theta)$ leads us to the following result.

\begin{lemma}\label{l:c1} For a fixed $\gamma>0$, $\beta\in\R$ and each $\Theta\in\mathbb{N}$, eigenvalues $\Omega_{n,\rho}$ and $\Omega_{n+\Theta,\rho}$ of the operator $\mathcal{Q}_0(\rho)$ collide for all $n\in\Z^\ast$ when $\beta>0$, and all $n\in(-\Theta,0)$ when $\beta\leq0$. For $\beta>0$, all collisions occur in an interval $(k^\ast, \infty)$ if $n\in(-\infty,-\Theta)\cup(0,\infty)$, in $(0, k^\ast)$ if $n\in(-\Theta,0)$, and for all $ k>0$ if $n\in (-2,0)$ where $k^\ast = (\g \mathscr{J}(n,\Theta)/\b \mathscr{U}(n,\Theta))^{1/4}$.
For $\beta\leq0$, all collisions occur for all $ k>0$. All collisions take place away from the origin in the complex plane except for the pair $\Omega_{n,\rho}$ and $\Omega_{-n,\rho}$. 
\end{lemma}
\begin{proof}
Note that
\begin{equation}
    \mathscr{J}(n,\Theta)=\left(n+\dfrac{\Theta}{2}\right)^2-\dfrac{\Theta^2-4}{4}.
\end{equation}
For a fixed $\Theta\in\mathbb{N}$, $\mathscr{J}(n_1,\Theta)=\mathscr{J}(n_2,\Theta)=0$ where \[n_1=-\frac{\Theta+\sqrt{\Theta^2-4}}{2}\quad \text{and}\quad n_2=-\frac{\Theta-\sqrt{\Theta^2-4}}{2}.\]
Note that $n_1, n_2$ are purely complex for $\Theta=1$, real and equal for $\Theta=2$, and real and distinct for $\Theta\geq 3$. Moreover, for $\Theta\geq 3$, $-\Theta<n_1<-\Theta+1/2$, and $-1/2<n_2<0$. Combining these, we have, $\mathscr{J}(n,\Theta)>0$ when $n\in(-\infty,-\Theta)\cup(0,\infty)$ for all $\Theta\in\N$, $\mathscr{J}(-1,2)=0$, and $\mathscr{J}(n,\Theta)<0$ when $n\in(-\Theta,0)$ for all $\Theta\geq3$.

Now, rewriting $\mathscr{U}(n,\Theta)$, we obtain
\begin{align*}
    \mathscr{U}(n,\Theta)=&3\left(n(n+\Theta)+\dfrac{\Theta^2-1}{6}\right)^2-\dfrac{(\Theta^2-1)^2}{12}\\
    =&\left(\left(n+\dfrac{\Theta}{2}\right)^2-\dfrac{\Theta^2}{4}\right)\left(3\left(n+\dfrac{\Theta}{2}\right)^2+\dfrac{\Theta^2}{4}-1 \right)\\
    =&:\mathscr{U}_1(n,\Theta)\mathscr{U}_2(n,\Theta).
\end{align*}
A root analysis similar to $\mathscr{J}(n,\Theta)$ on $\mathscr{U}_1(n,\Theta)$ and $\mathscr{U}_2(n,\Theta)$ provides that
$\mathscr{U}(n,\Theta)>0$ when $n\in(-\infty,-\Theta)\cup(0,\infty)$ for all $\Theta\in\N$, $ \mathscr{U}(-1,2)=0$, and $\mathscr{U}(n,\Theta)<0$ when $n\in(-\Theta,0)$ for all $\Theta\geq3$.
    
Note that, there exists a $\rho\in\R$ satisfying the collision condition \eqref{e:ccc1} if 
\begin{equation}\label{e:ccc5}
 X(n,\Theta):=-\g \mathscr{J}(n,\Theta)+\beta k^4 \mathscr{U}(n,\Theta)>0.
\end{equation}
We know that $\mathscr{J}(n,\Theta), \mathscr{U}(n,\Theta)>0$ when $n\in(-\infty,-\Theta)\cup(0,\infty)$ for all $\Theta\in\N$. The condition in \eqref{e:ccc5} does not hold for any $n\in(-\infty,-\Theta)\cup(0,\infty)$ when $\b<0$. For $\b>0$, \eqref{e:ccc5} holds for $n\in(-\infty,-\Theta)\cup(0,\infty)$ only when $k>k^\ast$. Since $\mathscr{J}(n,\Theta), \mathscr{U}(n,\Theta)\leq0$ when $n\in(-\Theta,0)$, there exist $\rho\in\R$ satisfying the collision condition \eqref{e:ccc1} for all $k>0$ when $\b\leq0$; and for all $k<k^\ast$ when $\b>0$. For instance, see Figures \ref{fig:1} and \ref{fig:2}. In Figure~\ref{fig:1} (A), we fix $\b=-1$, $\g=k=1$, $\Theta=4$,  then $X(n,\Theta)>0$ for $n\in(-4,0)$ and $X(n,\Theta)<0$ otherwise. In Figure~\ref{fig:1} (B) for $\b=\g=k=1$, $\Theta=4$, $X(n,\Theta)>0$ for all $n\in(-10,-5)$. In Figure~\ref{fig:2} (A), for $\b=\g=1$, $\Theta=3$, $n=-2$; \eqref{e:ccc5} holds for $k\in(0,0.707)$ and does not hold otherwise. Also, in Figure~\ref{fig:2} (B), for $\b=\g=1$, $\Theta=3$, $n=-5$; \eqref{e:ccc5} holds for $k\in(0.412,\infty)$ and does not hold otherwise.
Observe that, $\Omega_{n,\rho_c}=\Omega_{-n,\rho_c}=0$ for $\rho_c^2=(n^2-1)|\g-\b k^4n^2|$. Therefore, $\Omega_{n,\rho_c}$ and $\Omega_{n+\Theta,\rho_c}$ collide at the origin when $\Theta$ is even and $n=-\Theta/2$. All other collisions are away from origin.
\end{proof}

\begin{figure}%
    \centering
    
    \subfloat[\centering $X(n,4)$ with $n:-5$ to $1$, $\b=-1$, $\g=k=1$]{{\includegraphics[width=7cm]{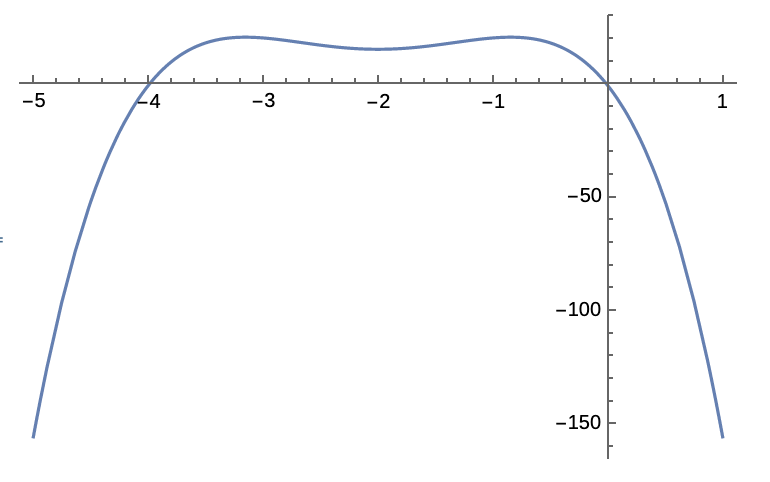} }}
    \qquad
    \subfloat[\centering $X(n,4)$ with $n:-10$ to $-5$, $\b=\g=k=1$ ]{{\includegraphics[width=7cm]{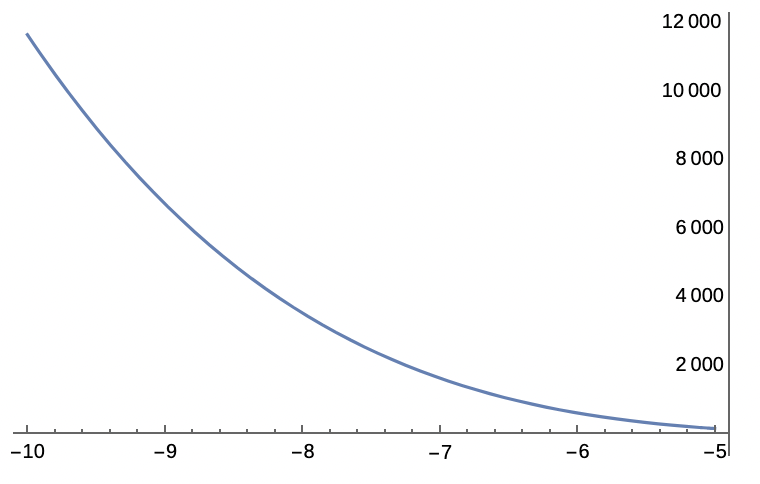} }}
    \caption{Graph of function $X(n,\Theta)$ vs. $n$ for $\Theta=4$}\label{fig:1}%
    \centering
    \subfloat[\centering $X(-2,3)$ with $k\in(0,0.9)$, $\g=\b=1$]{{\includegraphics[width=7cm]{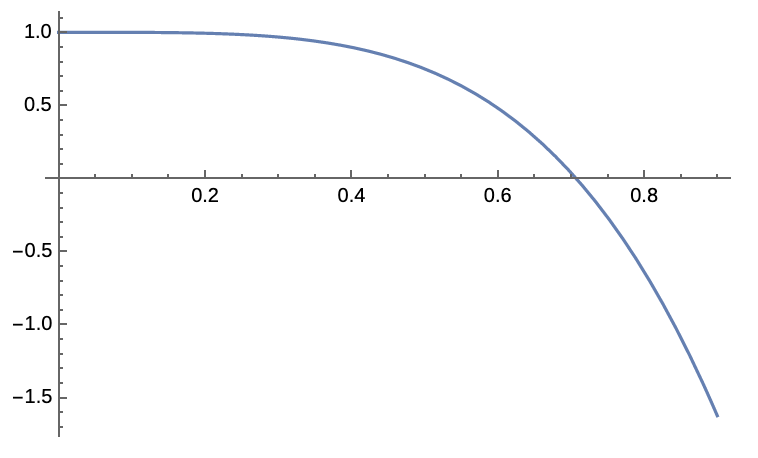} }}
    \qquad
    \subfloat[\centering $X(-5,3)$ with $k\in(0,0.5)$, $\g=\b=1$]{{\includegraphics[width=7cm]{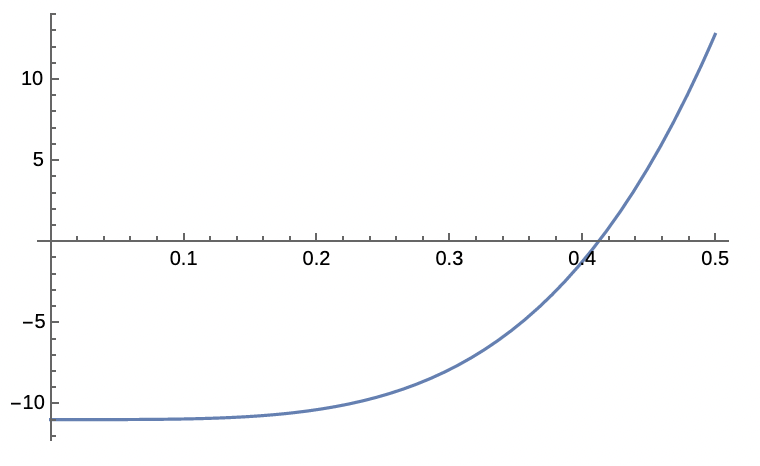} }}
    \caption{Graph of function $X(n,\Theta)$ vs. $k$ for $n=-2,-5$ and $\Theta=3$}%
    \label{fig:2}%
\end{figure}


For $|\varepsilon|$ sufficiently small, $\operatorname{spec}_{L_0^2(\T)}(\mathcal{Q}_\varepsilon(\rho))$ contain eigenvalues near to the eigenvalues of $\operatorname{spec}_{L_0^2(\T)}(\mathcal{Q}_0(\rho))$ which depend continuously upon $\varepsilon$.
The opposite sign of the {\em Krein signatures} of the colliding non-zero eigenvalues is a prerequisite for instability \cite{MacKay1986STABILITYWAVES}. The linear operator $\mathcal{Q}_\varepsilon(\rho)$ can be decomposed into the product of a skew-adjoint and a self-adjoint operator as follows \[
\mathcal{Q}_{\varepsilon}(\rho) = \mathcal{A} \B_{\varepsilon,\rho}
\]
where $\A=\partial_z$ is skew-adjoint and
\[
 \B_{\varepsilon,\rho}=k^2(c+\beta k^2\partial_z^2-2\eta)+(\g+\rho^2)\partial_z^{-2}.
\]
is self-adjoint.

An eigenvalue $\lambda\in i\R\setminus\{0\}$ has negative Krein
signature if $\left<\phi,\mathcal{Q}_\varepsilon(\rho)\phi\right><0$, where $\phi$ is the corresponding eigenfunction, and it has
positive Krein signature if 
$\left<\phi,\mathcal{Q}_\varepsilon(\rho)\phi\right>>0$. In particular, for sufficiently small $\varepsilon$, the Krein signature $\chi_{n,\rho}$ of eigenvalues $i\Omega_{n,\rho}$ in \eqref{E:oomega} of $\mathcal{Q}_{0}(\rho)$ is given by
\begin{align}\label{eq:krein}
    \chi_{n,\rho} = \operatorname{sgn}(\left<\B_{0,\rho} e^{inz}, e^{inz}\right>)=\operatorname{sgn}\left( \frac{1}{n}\Omega_{n,\rho}\right).
\end{align}
Here, $\operatorname{sgn}$ is the signum function which establishes the sign of a real number. If the collision condition \eqref{e:coll} is satisfied for some $n,m\in \Z^\ast$, then \eqref{eq:krein} states that at the collision, the non-zero eigenvalues $i\Omega_{n,\rho}$ and $i\Omega_{m,\rho}$ have opposite Krein signatures if
\begin{align}\label{eq:opp}
    n m<0
\end{align}
otherwise they have same Krein signatures at the collision. We rule out several collisions indicated in Lemma~\ref{l:c1} that won't cause transverse instability, using \eqref{eq:opp}.
\begin{lemma}[Potentially unstable nodes]\label{l:k1}Out of all the collisions mentioned in Lemma \ref{l:c1},
\begin{enumerate}
     \item For $\rho=0$, $\b\in\R$, $\{n,n+\Theta\}=\{-1,1\}$, are potentially unstable with respect to long wavelength transverse perturbations.
     \item For all $\b\in\R$; $\{n,n+\Theta\}$, $\rho\neq0$ with $n\in(-\Theta,0)$, $\Theta\geq3$ are potentially unstable with respect to finite or short wavelength transverse perturbations.
     
 \end{enumerate}
\end{lemma}

\subsubsection{\underline{Localized or bounded perturbations}}
As a consequence of the symmetry of the spectrum derived
in Proposition \ref{p:02}, we obtain instability if there is an eigenvalue of $\mathcal{Q}_\varepsilon(\rho,\xi)$ off the imaginary axis. We regard $\mathcal{Q}_\varepsilon(\rho,\xi)$ as a perturbation of the operator $\mathcal{Q}_0(\rho,\xi)$. Consider
\[\mathcal{Q}^\ast_\varepsilon(\rho,\xi)=\mathcal{Q}_\varepsilon(\rho,\xi)-\mathcal{Q}_0(\rho,\xi)
\]
This operator is compact in $L^2(\T)$,  with $\|\mathcal{Q}^\ast_\varepsilon(\rho,\xi)\|\to 0$ as $\varepsilon\to 0$ in the operator norm. Note that this estimate is uniform for $\xi\in(0,1/2]$. Therefore, spectra of $\mathcal{Q}_\varepsilon(\rho,\xi)$ and $\mathcal{Q}_0(\rho,\xi)$ remain close for sufficiently small $\varepsilon$.
To find the spectrum of  $\mathcal{Q}_\varepsilon(\rho,\xi)$, we require to find the spectrum of $\mathcal{Q}_0(\rho,\xi)$. A straightforward calculation reveals that 
\begin{align}\label{e:spec123}
    \mathcal Q_{0}(\rho,\xi)e^{inz} = i\Omega_{n,\rho,\xi}e^{inz}\quad \text{for all}\quad n \in \mathbb{Z}.
\end{align}
where \begin{align}\label{E:omega1}
    \Omega_{n,\rho,\xi} = \g\left((n+\xi)-\dfrac{1}{(n+\xi)}\right)+\b k^4((n+\xi)-(n+\xi)^3)-\dfrac{\rho^2}{(n+\xi)}
\end{align}
Therefore, the $L^2(\T)$-spectrum of $\mathcal Q_0(\rho,\xi), \xi\in(0,1/2]$ is given by
\begin{equation}\label{e:spec1}
    \operatorname{spec}_{L^2(\mathbb{T})}(\mathcal Q_0(\rho,\xi))=\{i\Omega_{n,\rho,\xi}; n \in \Z\}.
\end{equation}
When eigenvalues of $\mathcal Q_\varepsilon(\rho,\xi)$ bifurcate from the imaginary axis for $|\varepsilon|$ sufficiently small, according to the symmetry of the spectrum around the real and imaginary axes, they must do so in pairs as a result of collisions of eigenvalues of $\mathcal Q_0(\rho,\xi)$ on the imaginary axis.
Let $n\neq m\in \Z$, a pair of eigenvalues $i\Omega_{n,\rho,\xi}$ and $i\Omega_{m,\rho,\xi}$ of $\mathcal Q_0(\rho,\xi)$ collide for some $\rho=\rho_c$ and $\xi\in\left(0,1/2\right]$ when 
\begin{align}\label{e:colll3}
    \Omega_{n,\rho_c,\xi}=\Omega_{m,\rho_c,\xi}.
\end{align}
Without any loss of generality, consider $n<m$ and $m=n+\Theta$ with $\Theta\in\N$ in collision condition \eqref{e:colll3}, we arrive at
\begin{equation}\label{e:cc11}
    \rho^2=-\g \mathscr{F}(n,\xi,\Theta)+\beta k^4 \mathscr{V}(n,\xi,\Theta)
\end{equation}
where,
\[
\mathscr{F}(n,\xi,\Theta)=(n+\xi)(n+\xi+\Theta)+1\quad \text{and}
\]
\[
\mathscr{V}(n,\xi,\Theta)=3(n+\xi)^2(n+\xi+\Theta)^2+(n+\xi)(n+\xi+\Theta)(\Theta^2-1).
\]
\begin{lemma}\label{l:c2} For a fixed $\gamma>0$, $\beta\in\R$ and each $\Theta\in\mathbb{N}$, eigenvalues $\Omega_{n,\rho,\xi}$ and $\Omega_{n+\Theta,\rho,\xi}$ of the operator $\mathcal{Q}_0(\rho,\xi)$ collide for all $n\in\Z$ except $n=\{-2,-1\}$ for $\Theta=2$ when $\beta>0$, and all $n\in[-\Theta,-1]$ except $n=-1$ for $\Theta=1$ when $\beta\leq0$. All collisions take place away from the origin in the complex plane except when $\Theta$ is odd and $n=-(\Theta+1)/2$ in which case eigenvalues $\Omega_{n,\rho,\Theta}$ and $\Omega_{-n-1,\rho,\Theta}$ collide at the origin for $\rho=\rho_c(1/2)$. Table~\ref{tab:col_nonp} provides a range of values of $k$ and $\xi$ for all such collisions.
\begin{table}[ht]
    \centering
    \begin{tabular}{|c|c|c|c|c|}
    \hline
  $\beta$ & $\Theta$ & $n$ & $k$ & $\xi$ \\ \hline
        $\beta >0$ & $\mathbb{N}$ & $(-\infty,-\Theta)\cup(0,\infty)$ & $\left(\min(k_\xi),\infty\right)$ & $(0,1/2]$\\ \hline
        $\beta >0$ & $\mathbb{N}\setminus\{2\}$ &$(-\Theta,-1]$ & $\left(0,\max(k_\xi)\right]$ & $(0,1/2]$\\ \hline
        $\beta >0$ & $1$ &$-\Theta$ & $\left(\min (k_\xi),\infty\right)$ & $(0,1/2]$\\ \hline
        $\beta >0$ & $\geq 3$ &$-\Theta$ & $\left(0,\max(k_\xi)\right]$ & $\left((\Theta-\sqrt{\Theta^2-4})/2,1/2\right] $\\ \hline
       $\beta\leq 0$  & $2$ & $[-\Theta,-1]$ & $\left(\min (k_\xi),\infty\right)$ & $(0,1/2]$\\ \hline
       $\beta\leq 0$  & $\geq3$ & $(-\Theta,-1]$ & $(0,\infty)$ & $(0,1/2]$\\ \hline
       $\beta\leq 0$  & $\geq3$ & $-\Theta$ & $\left(\min(k_\xi),\infty\right)$ & $\left(0,(\Theta-\sqrt{\Theta^2-4})/2\right)$\\ \hline
       $\beta\leq 0$  & $\geq3$ & $-\Theta$ & $(0,\infty)$ & $\left[(\Theta-\sqrt{\Theta^2-4})/2,1/2\right] $\\ \hline
    \end{tabular}
    \caption{For a given sign of $\beta$ and value(s) of $\Theta$, each row lists value(s) of $n$ for which collisions takes place along with the value(s) of $k$ and $\xi$. Here $k_\xi=\left|\g \mathscr{F}(n,\xi,\Theta)/\b \mathscr{V}(n,\xi,\Theta)\right|^{1/4}.$}
    \label{tab:col_nonp}
\end{table}

\end{lemma}
\begin{proof}
 The function $\mathscr{F}(n,\xi,1),\mathscr{F}(n,\xi,2)>0$ for all $n\in\Z$, $\xi\in(0,1/2]$; while $\mathscr{F}(n,\xi,\Theta)$, $\Theta\geq3$ is zero for $n=-\Theta$, $\xi=\dfrac{\Theta-\sqrt{\Theta^2-4}}{2}$; positive for all $n\in(-\infty,-\Theta-1]\cup[0,\infty)$ for all $\xi\in(0,1/2]$ and $n=-\Theta$ for all $\xi\in\left(0,\dfrac{\Theta-\sqrt{\Theta^2-4}}{2}\right)$, while negative for all $n\in[-\Theta+1,-1]$ for all $\xi\in(0,1/2]$ and $n=-\Theta$ for all $\xi\in\left(\dfrac{\Theta-\sqrt{\Theta^2-4}}{2},\dfrac{1}{2}\right]$. The function $\mathscr{V}(n,\xi,1)>0$ for all $n\in\Z$; and $\mathscr{V}(n,\xi,\Theta)$, $\Theta\geq2$ is positive for all $n\in(-\infty,-\Theta-1]\cup[0,\infty)$ and negative for all $n\in[-\Theta,-1]$ for all $\xi\in(0,1/2]$.
There exist $\rho\in\R$ satisfying the collision condition \eqref{e:colll3} if 
\begin{equation}\label{e:cc5}
Z(n,\xi,\Theta)=-\g \mathscr{F}(n,\xi,\Theta)+\beta k^4 \mathscr{V}(n,\xi,\Theta)>0
\end{equation}
which holds for all $n\in\Z$ except $n=\{-2,-1\}$ for $\Theta=2$ when $\b>0$; and for all $n\in[-\Theta,-1]$ except $n=-1$ for $\Theta=1$ when $\beta\leq0$. For instance, see Figure \ref{fig:np1}. Fixing $\b=-1$, $\g=k=1$; $Z(n,0.4,5)>0$ for $n\in[-5,-1]$ and $Z(n,0.4,5)<0$ otherwise. Also, for $\b=\g=k=1$; $Z(n,0.3,2)>0$ for all $n$ except $n=-2,-1$. Now see Figure \ref{fig:np2}. For $\b=\g=1$, $\Theta=3$, $n=-2$; \eqref{e:cc5} holds for $k\in(0,0.811)$ and does not hold otherwise. Also for $\b=\g=1$, $k=0.2$, $n=-\Theta=-4$; \eqref{e:cc5} holds for $\xi\in(0.268,0.5)$ and does not hold otherwise, which agrees with the calculation. Note that $\Omega_{n,\rho,\xi}=0$ at $\rho^2=|((n+\xi)^2-1)(\g-\b k^4(n+\xi)^2)|$. $\Omega_{n,\rho_c,\xi}=\Omega_{n+\Theta,\rho_c,\xi}$=0 for a fixed $\rho_c$ is possible only for $\Theta=-2n-1$, $\xi=1/2$. Therefore, $\Omega_{n,\rho_c,\xi}$ and $\Omega_{n+\Theta,\rho_c,\xi}$ collide at the origin for $n=-(\Theta+1)/2$, for all $n\in[-\Theta,-1]\cap \mathbb{Z}$, $\xi=1/2$ and $\rho_c^2=\left|\left(n^2+n-\dfrac{3}{4}\right)\left(\g-\b k^4\left(n+\dfrac{1}{2}\right)^2\right)\right|$; except the pair $\{n,n+\Theta\}=\{-1,0\}$. All other collisions are away from the origin.

\end{proof}
\begin{figure}%
    \centering
    
    \subfloat[\centering $Z(n,0.4,5)$ with $n:-6$ to $1$, $\b=-1$, $\g=k=1$]{{\includegraphics[width=7cm]{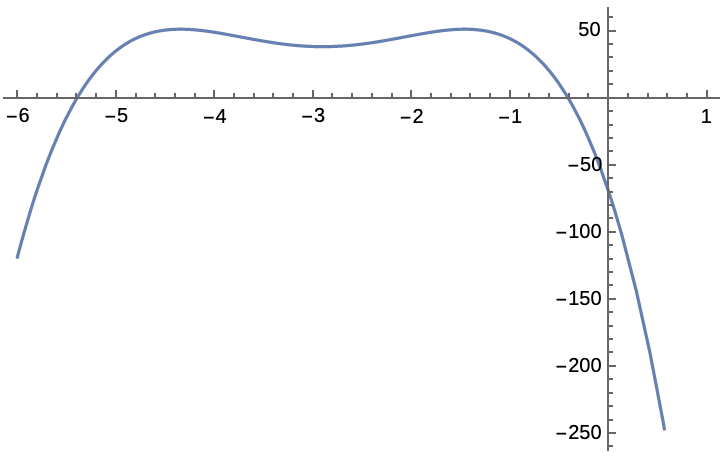} }}
    \qquad
    \subfloat[\centering $Z(n,0.3,2)$ with $n:-3$ to $0$, $\b=\g=k=1$ ]{{\includegraphics[width=7cm]{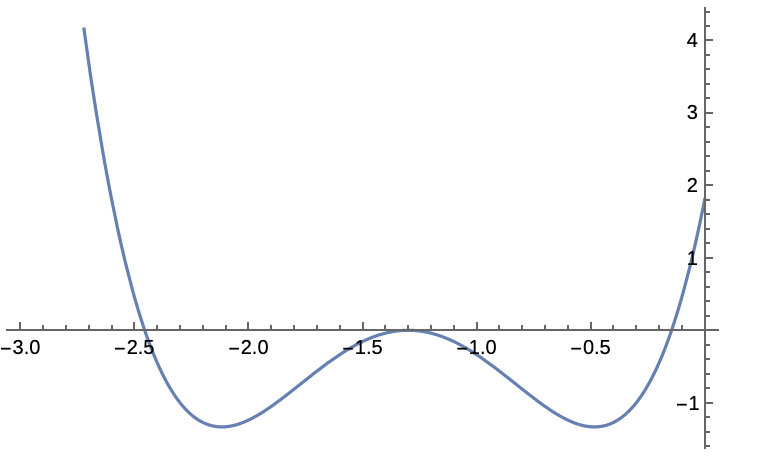} }}
    \caption{Graph of function $Z(n,\xi,\Theta)$ vs. $n$ }\label{fig:np1}%
    \centering
    \subfloat[\centering $Z(-2,0.4,3)$ with $k\in(0,1.5)$, $\g=\b=1$]{{\includegraphics[width=7cm]{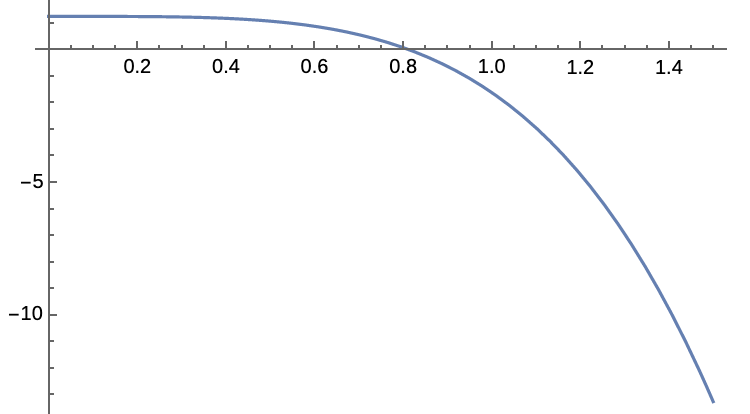} }}
    \qquad
    \subfloat[\centering $Z(-4,\xi,4)$ with $\xi\in(0,0.5)$, $k=0.2$, $\g=\b=1$]{{\includegraphics[width=7cm]{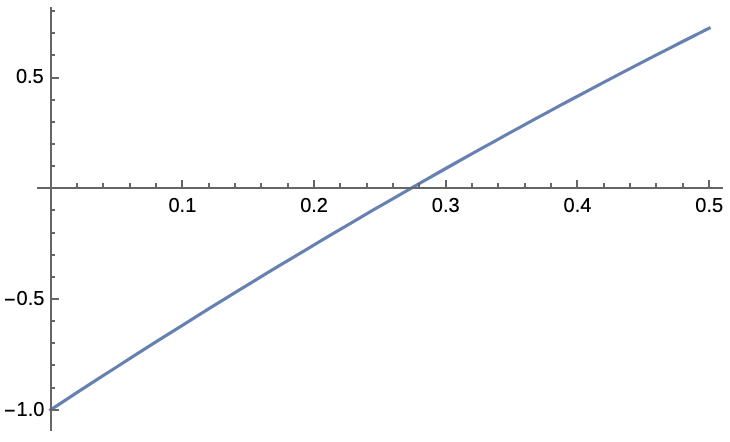} }}
    \caption{Left plot: Graph of function $Z(n,\xi,\Theta)$ vs. $k$, \\ Right plot: Graph of function $Z(n,\xi,\Theta)$ vs. $\xi$.}%
    \label{fig:np2}%
\end{figure}


The linear operator $\mathcal Q_\varepsilon(\rho,\xi)$ can be decomposed similarly to the preceding section
\[
\mathcal Q_\varepsilon(\rho,\xi) = \mathcal{A}_\xi \B_\varepsilon(\rho,\xi)
\]where
\[\mathcal{A}_\xi = \partial_z+i\xi \quad \text{and} \quad \B_\varepsilon(\rho,\xi) = k^2(c+\beta k^2(\partial_z+i\xi)^2-2\eta)+(\g+\rho^2)(\partial z+i\xi)^{-2}
\]
The operator $\mathcal A_\xi$ is skew-adjoint, whereas the operator $\B_\varepsilon(\rho,\xi)$ is self-adjoint. Using the definition in \eqref{eq:krein}, the Krein signature, $\chi_{n,\xi}$ of an eigenvalue $i\Omega_{n,\rho,\xi}$ in $\operatorname{spec}(\mathcal Q_0(\rho,\xi)))$ is
\begin{equation}\label{e:krsig}
\chi_{n,\xi} = \operatorname{sgn}\left( \dfrac{1}{n}\Omega_{n,\rho,\xi}\right),\quad n \in \Z. 
\end{equation}
Consequently, $(n+\xi_0)(m+\xi_0)< 0$, where $\xi_0$ is the value of the floquet exponent where eigenvalues $i\Omega_{n,\rho,\xi}$ and $i\Omega_{m,\rho,\xi}$ collide, is a necessary condition for the bifurcation of colliding eigenvalues from the imaginary axis.
\begin{lemma}[Potentially unstable nodes]Out of all the collisions mentioned in Lemma \ref{l:c2}, all $n\in[-\Theta,-1]$ for each $\Theta\in\N\setminus\{2\}$ when $\b>0$, and all $n\in[-\Theta,-1]$ for each $\Theta\in\N\setminus\{1\}$ when $\b\leq0$ are potentially unstable with respect to finite or short wavelength transverse perturbations.

\end{lemma}
\begin{table}[ht]
    \centering
    \begin{tabular}{|c|c|c|c|c|c|c|c|c|}
    \hline
       $\boldsymbol{\Theta}$  & \multicolumn{4}{c|}{\textbf{LWTP}}& \multicolumn{3}{c|}{\textbf{F/SWTP}} \\\hline
       & \multicolumn{2}{c|}{\textbf{Periodic}} & \multicolumn{2}{c|}{\textbf{Non-periodic}} & \textbf{Periodic} & \multicolumn{2}{c|}{\textbf{Non-periodic}} \\\hline
       & $\boldsymbol{\beta>0}$ & $\boldsymbol{\beta\leq0}$ & $\boldsymbol{\beta>0}$ & $\boldsymbol{\beta\leq0}$ & $\boldsymbol{\beta\in\R}$ & $\boldsymbol{\beta>0}$ & $\boldsymbol{\beta\leq0}$ \\\hline
       $1$ & None & None & None & None & None & $\{-1,0\}$ & None \\\hline
       $2$ & $\{-1,1\}$ & $\{-1,1\}$ & None & None & None & None & $\{-1,1\}$ \\
       ~&~&~&~&~&~&~&$, \{-2,0\}$\\\hline
       $\geq3$ & None & None & None &None& $\{-1,\Theta-1\}$&$\{-1,\Theta-1\}$&$\{-1,\Theta-1\}$\\
       ~& ~&~& ~&~&$\{-2,\Theta-2\},$&$\{-2,\Theta-2\},$&$\{-2,\Theta-2\},$\\
       ~& ~&~& ~&~&$,\dots$&$,\dots$&$,\dots$\\
       ~& ~&~&~&~&$\{-\Theta+1,1\}$&$\{-\Theta+1,1\}$&$\{-\Theta+1,1\}$\\
       ~&~&~&~&~&~&$\{-\Theta,0\}$&$\{-\Theta,0\}$\\\hline
    \end{tabular}
    \caption{Potentially unstable nodes for a given $\Theta\in\mathbb{N}$. Here `LWTP' stands for {\em Long wavelength transverse perturbations} and `F/SWTP' stands for {\em Finite or short wavelength transverse perturbations}. }
    \label{tab:col}
\end{table}
\subsection{Long wavelength transverse perturbations}
We start the further analysis with the values of $\rho$ sufficiently close to origin, that is, $|\rho|\leq\rho_0
$ for some $\rho_0>0$. Table \ref{tab:col} summarizes all the potentially unstable nodes with respect to long wavelength transverse perturbations. 
For $\varepsilon$ and $\rho$ sufficiently small, $\{\Omega_{-1,\rho},\Omega_{1,\rho}\}$ are pair of eigenvalues bifurcating continuously from $\{\Omega_{-1,0},\Omega_{1,0}\}$. For $\varepsilon=0$, $\{\Omega_{-1,0},\Omega_{1,0}\}$ are equipped with eigenfunctions $\{e^{-iz},e^{iz}\}$. We choose the real basis $\{\cos z,\sin z\}.$ We calculate expansion of a basis $\{\psi_1,\psi_2\}$ for the eigenspace corresponding to the eigenvalues of $\{\Omega_{-1,\rho},\Omega_{1,\rho}\}$ in $L^2_0(\mathbb{T})$ by using expansions of $\eta$ and $c$ in \eqref{e:expptw}, for small $\varepsilon$ and $\rho$ as
\begin{align*}
	\psi_{1}(z)  & =\cos z+2 \varepsilon A_{2} \cos 2 z+O(\varepsilon^{2}),\\
	\psi_{2}(z) & =\sin z+2 \varepsilon A_{2} \sin 2 z+O(\varepsilon^{2}).
\end{align*} 
We have the following expression for  $\mathcal{Q}_{\varepsilon}(\rho)$ after expanding and using $\eta$ and $c$
\begin{equation}\label{e:ex}
    \mathcal{Q}_{\varepsilon}(\rho)= \g(\partial_z+\partial_z^{-1})+\rho^2\partial_z^{-1}+\b k^4(\partial_z+\partial_z^3)-2\varepsilon k^2\partial_z(\cos z)+O(\varepsilon^2)
\end{equation}
In order to locate the bifurcating eigenvalues for $|\varepsilon|$ sufficiently small, we calculate the action of $\mathcal{Q}_{\varepsilon}(\rho)$ on the extended eigenspace $\{\psi_1(z), \psi_2(z)\}$ viz.
\begin{align}\label{eq:bmat1}
    \mathcal{T}_{\varepsilon}(\rho) = \left[ \frac{\langle \mathcal{Q}_\varepsilon(\rho)\psi_{i}(z),\psi_{j}(z)\rangle}{\langle\psi_{i}(z),\psi_{i}(z)\rangle} \right]_{i,j=1,2}
\text{ and }
    \mathcal{I}_{\varepsilon} = \left[ \frac{\langle \psi_{i}(z),\psi_{j}(z)\rangle}{\langle\psi_{i}(z),\psi_{i}(z)\rangle} \right]_{i,j=1,2}.
\end{align}
We use expansion of $\mathcal Q_\varepsilon(\rho)$ in \eqref{e:ex} to find actions of $\mathcal{Q}_\varepsilon(\rho)$ and identity operator on $\{\psi_1,\psi_2\}$, and arrive at
\begin{align*}
	\mathcal{T}_\varepsilon(\rho)&=\begin{pmatrix}
	0 & \rho^2+2\varepsilon^2k^2\eta_2 \\ -\rho^2 & 0\end{pmatrix}+O(\varepsilon^2(\rho+\varepsilon)),\\
	\end{align*}
	To locate where these two eigenvalues are bifurcating from the origin, we analyze the characteristic equation $	\left|\mathcal{T}_\varepsilon(\rho)-\lambda \mathcal{I}\right|=0$, where $\mathcal{I}_\varepsilon$ is $2\times 2$ identity matrix. From which we conclude that
	\begin{equation}
	    \lambda^2+\rho^2(\rho^2+2\varepsilon^2k^2\eta_2)=O(|\varepsilon|\rho^2(\rho^2+\varepsilon^2))
	\end{equation}
	From which we conclude that
	\begin{equation}
	    \lambda^2=-\rho^2(\rho^2+2\varepsilon^2k^2\eta_2)+O(|\varepsilon|\rho^2(\rho^2+\varepsilon^2))
	\end{equation}
	For $\rho=\varepsilon=0$, we obtain zero as a double eigenvalue, which is consistent with our calculation. For $\beta<0$, we obtain two purely imaginary eigenvalues for all $\rho$ and $\varepsilon$ sufficiently small. For $\b>0$, we obtain two real eigenvalues with opposite signs when
	\begin{equation}
	    k>\sqrt[4]{\dfrac{\g}{4\b}}
	\end{equation}
 Hence the Theorem \ref{t:1}.

\subsection{Finite or short wavelength transverse perturbations}
Here we work with the values of $\rho$ away from the origin, $|\rho|\geq\rho_0$ for some $\rho_0>0$. We do further analysis to check if collisions in Table \ref{tab:col} indeed lead to instability or not.
\subsubsection{\underline{(In)stability analysis for $\Theta=1$}:}
For periodic perturbations, there are no collisions mentioned for $\Theta=1$. However, there in one pair of colliding eigenvalues $\{\Omega_{-1,\rho,\xi}, \Omega_{0,\rho,\xi}\}$ for $\b>0$ with respect to non-periodic perturbations, colliding for
\begin{equation}\label{e:rc}
    \rho^2=-\g(1-\xi+\xi^2)+3\b k^4\xi^2(1-\xi)^2 \quad \text{and}\quad k>min\left(\dfrac{\g(1-\xi+\xi^2)}{3\b \xi^2(1-\xi)^2}\right)^{1/4}
\end{equation}
There exists a curve $\rho=\rho_c$ given in \eqref{e:rc} along which \[\Omega(\rho_c,\xi):=\Omega_{-1,\rho_c,\xi} =\Omega_{0,\rho_c,\xi}.\]
Furthermore,
\begin{align}
 \phi_{0,-1}(z) = e^{-i z}
 \quad \quad and \quad\quad
 \phi_{0,0}(z) = 1
\end{align}
forms the corresponding eigenspace for $\mathcal{Q}_0(\rho_c,\xi)$ associated with the two eigenvalues. Let
\begin{align}
    i \Omega(\rho_c,\xi) + i \theta_{\varepsilon,\rho,-1}
    \quad \quad and \quad\quad
    i \Omega(\rho_c,\xi) + i \theta_{\varepsilon,\rho,0}
\end{align}
be the eigenvalues of $\mathcal{Q}_\varepsilon(\rho_c,\xi)$ bifurcating from $i\Omega_{-1,\rho_c,\xi}$ and $i\Omega_{0,\rho_c,\xi}$ respectively for $|\varepsilon|$ and $|\rho-\rho_c|$ small. Let $\{\phi_{\varepsilon,\rho,-1}(z), \phi_{\varepsilon,\rho,0}(z)\}$ be the extended eigenspace associated with two bifurcating eigenvalues. Following \cite{Creedon2021High-FrequencyApproach}, we can take,
\begin{align}\label{eq:eigg01}
    \phi_{\varepsilon,-1,\rho}(z) =& e^{-iz}+\varepsilon \phi_{-1}+O(\varepsilon^2), \\
    \phi_{\varepsilon,0,\rho}(z) =& 1+\varepsilon\phi_{0}+O(\varepsilon^2)\label{eq:eigg02}
\end{align}
Using orthonormality conditions on the eigenfunctions $\phi_{\varepsilon,\rho,-1}$ and $\phi_{\varepsilon,\rho,0}$, we get
\[
\phi_{-1}=\phi_{0} = 0.
\]
In order to find eigenvalues, we compute matrix representations of $\mathcal{Q}_\varepsilon(\rho_c,\xi)$ and identity operators on $\{\phi_{\varepsilon,-1,\rho}(z), \phi_{\varepsilon,0,\rho}(z)\}$ for sufficiently small $|\varepsilon|$ and $|\rho-\rho_c|$ 
\[
\mathcal{B}_\varepsilon(\rho,\xi) = 
\begin{pmatrix}i\Omega(\rho_c,\xi)-i\dfrac{\varsigma}{\xi-1}&
-i \varepsilon k^2\xi
\\
-i\varepsilon k^2(\xi-1) &i\Omega(\rho_c,\xi)-i\dfrac{\varsigma}{\xi}
\end{pmatrix}+O(|\varepsilon|(|\varsigma|+|\varepsilon|)),
\]
where $\varsigma=\rho^2-\rho_c^2$ and
$\mathcal{I}_\varepsilon$ is the $2\times 2$ identity matrix. Calculating the characteristic equation $\det(\mathcal{B}_\varepsilon(\rho,\xi)-\lambda\mathcal{I}_\varepsilon)=0$ for $\lambda$ of the form
\[
\lambda = i\Omega(\rho_c,\xi) +i\theta,
\]
leads to the polynomial equation
\[
\theta^2+\theta\varsigma\left(\dfrac{1}{\xi-1} + \frac{1}{\xi} +
O(\varepsilon^2)\right)
-\varepsilon^2k^4\xi(\xi-1)+\frac{\varsigma^2}{\xi(\xi-1)} +
  O(\varepsilon^2(|\varsigma|+\varepsilon^2)) =0.
\]
A direct computation shows that the discriminant of this polynomial is
\[
\operatorname{disc}_\varepsilon(\varsigma,\xi) = 
\frac{\varsigma^2}{\xi^2(\xi-1)^2} +4k^4\xi(\xi-1)\varepsilon^2+
O(\varepsilon^2(|\varsigma|+\varepsilon^2)).
\]For any $\varepsilon$ sufficiently small there exists
\[
\varsigma_\varepsilon(\xi) = 2k^2\xi^{3/2}(1-\xi)^{3/2}|\varepsilon|+O(\varepsilon^2)>0
\]
such that the two eigenvalues of $\mathcal Q_\varepsilon(\rho,\xi)$  are purely imaginary when $|\rho^2-\rho_c^2|\geq \varsigma_\varepsilon(\xi)$ and complex with opposite nonzero real parts when $|\rho^2-\rho_c^2|<\varsigma_\varepsilon(\xi)$, which proves the Theorem~\ref{t:2}.

The instability result in Theorem~\ref{t:2} is also supported by numerical experiments.  We consider the eigenvalue problem~\eqref{e:non} with operator $\mathcal Q_{\varepsilon}(\rho,\xi)$ defined by~\eqref{e:op7}. We use the shift-invert technique~\cite{saad1992numerical} and consider the following problem,
\begin{equation}\label{e:shiftinvert}
    \left[\mathcal Q_{\varepsilon}(\rho,\xi)\varphi-i\omega\right]^{-1}=\frac{1}{\lambda-i\omega}\varphi;\quad\varphi\in L^2(\T)\quad\text{and}\quad\xi\in(0,1/2].
\end{equation}
Here $\omega \in \mathbb{R}$ and is a guess chosen close to but not equal to $\lambda$. We use the MINRES~\cite{trefethen1997numerical} method to invert the operator iteratively.
The eigenvalue problem is solved numerically by \textbf{eigs} function using MATLAB which is a matrix-free function based on Arnoldi iterations. This allows us to have a $O(N\log N)$ flops scheme instead of $O(N^2)$ operations that would be needed if we would form an operator matrix.
\begin{figure}[ht]
    \centering
    \includegraphics[width=0.49\textwidth]{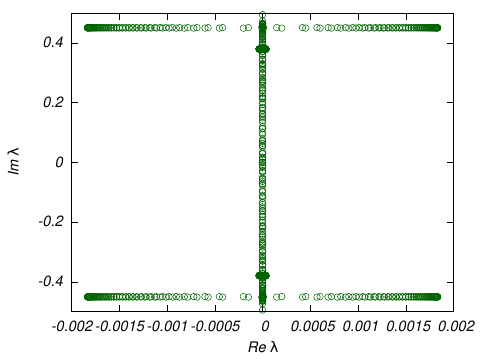}
    \includegraphics[width=0.49\textwidth]{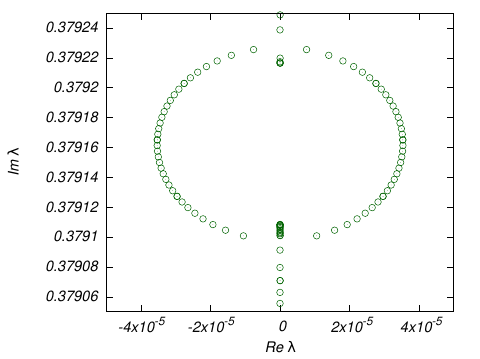}
    \caption{Left panel: Eigenvalue Spectrum of $\mathcal{Q}_\varepsilon(\rho,\xi)$ with $\gamma=1$, $\beta=1$, $k=2>4^{1/4}$ with the amplitude of the initial condition being $\epsilon = 0.01$. \\
    Right panel: zoom into a high-frequency bubble centered around $0.37916 i$ on the imaginary axis.}
    \label{fig:spectrum}
\end{figure}

\begin{figure}[ht]
    \centering
    \includegraphics[width=0.49\textwidth]{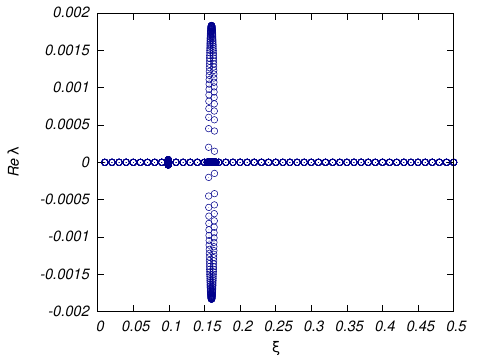}
    \caption{$\xi$ vs. $\Re(\lambda)$ of $\mathcal{Q}_\varepsilon(\rho,\xi)$ with $\gamma=1$, $\beta=1$, $k=2>4^{1/4}$ with the amplitude of the initial condition being $\epsilon = 0.01$.}
    \label{fig:xi}
\end{figure}

The spectrum of $\mathcal{Q}_\varepsilon(\rho,\xi)$ with $\gamma=1$, $\beta=1$, and $k=2>4^{1/4}$ is shown in Figure~\ref{fig:spectrum}. High-frequency bubbles appear along the imaginary axis for the solution~\eqref{e:expptw} with the amplitude of the initial condition chosen to be $\epsilon = 0.01$. In the right panel of Figure~\ref{fig:spectrum}, we zoom into one of the bubbles centered around $0.37916 i$ on the imaginary axis. 
In Figure~\ref{fig:xi}, we show the range of parameter $\xi$ for two high-frequency bubbles located above the imaginary axis.

\subsubsection{\underline{(In)stability analysis for $\Theta\geq2$}:}
For some $n \in \mathbb{Z}$ and a fixed $\Theta\geq2$, we have
\begin{equation}\label{e:1}
 i \Omega_{n, \rho_\ast,\xi}=i \Omega_{n+\Theta, \rho_\ast,\xi}=i \Omega(\rho_\ast,\xi), \quad \xi\in (-1/2,1/2]
\end{equation}
$\xi=0$ represents the periodic case and $\xi\neq0$ represents the non-periodic case. $i\Omega$ is an eigenvalue of $\mathcal{Q}_0(\rho_c,\xi)$ of multiplicity two with an orthonormal basis of eigenfunctions $\{e^{inz},e^{i(n+\Theta)z}\}$. Let $i \Omega + i \nu_{\varepsilon,n}$ and $i \Omega + i \nu_{\varepsilon,n+2}$ be the eigenvalues of $\mathcal{Q}_\varepsilon(\rho,\xi)$ bifurcating from $i\Omega_{n,\rho_c,\xi}$ and $i\Omega_{n+2,\rho_c,\xi}$ respectively, for $|\varepsilon|$ and $|\rho-\rho_c|$ small. Let $\{\varphi_{\varepsilon,n}(z), \varphi_{\varepsilon,n+2}(z)\}$ be a orthonormal basis for the corresponding eigenspace. We suppose the following expansions 
\begin{align}\label{eq:eiggg01}
    \varphi_{\varepsilon,n}(z) =& e^{inz}+\varepsilon\varphi_{n,1}+\varepsilon^2\varphi_{n,2}+\dots+\varepsilon^\Theta\varphi_{n,\Theta}+O(\varepsilon^{\Theta+1}), \\
    \varphi_{\varepsilon,n+\Theta}(z) =& e^{i(n+\Theta)z}+\varepsilon\varphi_{n+\Theta,1}+\varepsilon^2\varphi_{n+\Theta,2}+\dots+\varepsilon^\Theta\varphi_{n+\Theta,\Theta}+O(\varepsilon^{\Theta+1}).\label{eq:eiggg02}
\end{align}
We use orthonormality of $\varphi_{\varepsilon,n}$ and $\varphi_{\varepsilon,n+\Theta}$ to find that
\[
\varphi_{n,1} = \varphi_{n,2} =\dots=\varphi_{n,\Theta}= \varphi_{n+\Theta,1} = \varphi_{n+\Theta,2} = \dots = \varphi_{n+\Theta,\Theta}=0.\]
Using the expansions of $\eta$ and $c$ in \eqref{e:expptw}, we expand $\mathcal{Q}_\varepsilon(\rho,\xi)$ in $\varepsilon$ as
\begin{equation}
\begin{aligned}\label{e:op111}
    \mathcal{Q}_\varepsilon(\rho,\xi)=\mathcal{Q}_0(\rho,\xi)+k^2(\b_2\varepsilon^2+\b_4\varepsilon^4+\dots)(\partial_z+i\xi)&-2k^2\varepsilon\delta_1(\partial_z+i\xi)(\cos z)\\&-\dots-2k^2\varepsilon^\Theta\delta_\Theta(\partial_z+i\xi)(\cos (\Theta z))
\end{aligned}
\end{equation}
To explicitly obtain the values of all unknown coefficients in the expansion of $\mathcal{Q}_\varepsilon(\rho,\xi)$, we require coefficients of higher powers of $\varepsilon$ in the expansion of solution $\eta$. Calculating higher coefficients is difficult as the coefficients of the solution do not seem to have any apparent symmetry. Therefore, we pursue the instability analysis without calculating the unknown coefficients explicitly.

Following the same procedure as in the preceding subsection, we arrive at
\begin{align*}
	\mathscr{G}_\varepsilon(\rho,\xi)&=\begin{pmatrix}
		\mathscr{G}_{11} & \ \mathscr{G}_{12} \\ \mathscr{G}_{21} & \mathscr{G}_{22}\end{pmatrix}+O(\varepsilon^{\Theta+1}),\\
	\end{align*}
where 
\begin{align*}
		\mathscr{G}_{11}& =i\Omega(\rho_\ast,\xi)-i\dfrac{\varsigma}{n+\xi}+ik^2(n+\xi)(\b_2\varepsilon^2+\b_4\varepsilon^4+\dots)  , \\
	\mathscr{G}_{12} & =-ik^2(n+\xi+\Theta)\delta_\Theta \varepsilon^\Theta,\\
	\mathscr{G}_{21}&=-ik^2(n+\xi)\delta_\Theta \varepsilon^\Theta,\\ \mathscr{G}_{22}&=i\Omega(\rho_\ast,\xi)-i\dfrac{\varsigma}{n+\xi+\Theta}+ik^2(n+\xi+\Theta)(\b_2\varepsilon^2+\b_4\varepsilon^4+\dots),
\end{align*} 
The resulting discriminant of the characteristic equation $\det(\mathscr{G}_{\varepsilon}(\rho,\xi)-(i \Omega(\rho_\ast,\xi) + i \theta) \mathcal{I}_{\varepsilon}) = 0$ is
\begin{align*}
    \operatorname{disc}_\varepsilon(\varsigma) = \dfrac{\Theta^2\varsigma^2}{(n+\xi)^2(n+\xi+\Theta)^2}+k^4\Theta^2\beta_2^2\varepsilon^4+O(\varepsilon^2(|\varsigma|+|\varepsilon^3|)).
\end{align*}
For sufficiently small $|\varsigma|$ and $|\varepsilon|$ $\operatorname{disc}_\varepsilon(\varsigma)$ is positive which implies that no eigenvalue of $\mathcal{Q}_\varepsilon(\rho,\xi)$ is bifurcating from the imaginary axis due to collision. Hence the proof of Theorems \ref{t:3} and \ref{t:4}.

\section{Generalized RMKP} \label{sec:gRMKP}
\subsection{Construction of the spectral problem}
Linearizing the gRMKP equation \eqref{e:gRMKP} about its one-dimensional periodic traveling wave $\eta_\mathfrak{g}$ in \eqref{E:w_ansatz} and seeking a solution of the form
\begin{equation}\label{e:per1}
\eta_\mathfrak{g}^\ast(z,y,t)=e^{\frac{\lambda}{k} t+i\rho y}\varphi(z),\quad \lambda\in \mathbb{C},~ \rho\in\R 
\end{equation}, we arrive at
\begin{equation}\label{e:opt1}
\mathcal G^\mathfrak{g}_a(\lambda,\rho) \varphi:=(\lambda\partial_z- k^2\partial_z^2(c-\J_k-2\alpha_1\eta_\mathfrak{g}-3\alpha_2\eta^2_\mathfrak{g})-\rho^2-\g)\varphi=0
\end{equation}
\subsubsection{\underline{Periodic perturbations}}\label{ss:g11}
The invertibility problem
\[\mathcal G^\mathfrak{g}_{\varepsilon}(\lambda,\rho)\varphi=0;\quad \varphi\in L^2(\T)\]
is transformed into a spectral problem which requires invertibility of $\partial_z$. Since $\partial_z$ is not invertible in $L^2(\mathbb{T})$, we restrict the problem to mean-zero subspace $L^2_0(\mathbb{T})$, defined in \eqref{e:zerom}, of $L^2(\mathbb{T})$. 
\begin{lemma}
$\mathcal{G}^\mathfrak{g}_\varepsilon(\lambda,\rho)$ acting in $L^2(\T)$ with domain $H^{\mathfrak{b}+2}(\T)$ is not invertible if and only if $\lambda$ belongs to $L_0^2(\mathbb{T})$-spectrum of the operator $\mathcal Q_{\varepsilon}(\rho)$, where 
\begin{equation}\label{e:op10}
	\mathcal Q^\mathfrak{g}_{\varepsilon}(\rho):= k^2\partial_z(c- \J_k-2\alpha_1\eta_\mathfrak{g}-3\alpha_2\eta^2_\mathfrak{g})+(\g+\rho^2)\partial z^{-1}.\\
\end{equation}
\end{lemma}
\begin{proof}
The proof is similar to the proof of Lemma \ref{l:1}
\end{proof}
We arrive at pseudo-differential spectral problem
\begin{equation}\label{e:sp122}
  \mathcal Q^\mathfrak{g}_{\varepsilon}(\rho)\varphi=\lambda\varphi,
\end{equation}where $\varphi\in L_0^2(\T)$. With respect to periodic perturbations, we will study the spectrum of the operator $\mathcal Q^\mathfrak{g}_{\varepsilon}(\rho)$ acting in  $L^2_0(\T)$ with domain $H^{\mathfrak{b}+1}(\T))\cap L^2_0(\T)$. 
\begin{proposition}\label{p:3}
The operator $\mathcal Q^\mathfrak{g}_{\varepsilon}(\rho)$ possess following properties.
\begin{enumerate}
    \item The operator $\mathcal Q^\mathfrak{g}_{\varepsilon}(\rho)$ commutes with the reflection through the real axis.
    \item The operator $\mathcal Q^\mathfrak{g}_{\varepsilon}(\rho)$ anti-commutes with the reflection through the origin and the imaginary axis.
    \item The spectrum of $\mathcal Q^\mathfrak{g}_{\varepsilon}(\rho)$ is symmetric with respect to the reflections through origin, real axis and imaginary axis.
\end{enumerate}
\end{proposition} 
\begin{proof}
The proof is similar to the proof of Proposition \ref{p:1}
\end{proof}
\subsubsection{\underline{Localized or bounded perturbations}}\label{ss:g12}
In $L^2(\R)$ or $C_b(\R)$, the operator $\mathcal G^\mathfrak{g}_\varepsilon(\lambda, \rho)$ now have continuous spectrum. We can use the Floquet theory such that all solutions of \eqref{e:opt1} in $L^2(\R)$ or $C_b(\R)$ are of the form $\varphi(z)=e^{i\xi z}\Tilde{\varphi}(z)$ where $\xi\in\left(-1/2,1/2\right]$ is the Floquet exponent and $\Tilde{\varphi}$ is a $2\pi$-periodic function. We can deduce that the study of the invertibility of $\mathcal G^\mathfrak{g}_\varepsilon(\lambda,\rho)$ in $L^2(\R)$ or $C_b(\R)$ is equivalent to the invertibility of the linear operators $\mathcal G^\mathfrak{g}_{\varepsilon,\xi}(\lambda,\rho)$ in $L^2(\mathbb{T})$ with domain $H^{\mathfrak{b}+2}(\mathbb{T})$, for all $\xi\in\left(-1/2,1/2\right]$, where
\begin{align*}\label{E:bloch}
\mathcal G^\mathfrak{g}_{\varepsilon,\xi}(\lambda,\rho) = \lambda(\partial_z+i\xi)- k^2(\partial_z+i\xi)^2(c-\J_k-2\alpha_1\eta_\mathfrak{g}-3\alpha_2\eta^2_\mathfrak{g})-\rho^2-\g
\end{align*}
Also, $\mathcal G^\mathfrak{g}_{\varepsilon,\xi}(\lambda,\rho)$ is not invertible in $L^2(\mathbb{T})$ if and only if zero is an eigenvalue of $\mathcal G^\mathfrak{g}_{\varepsilon,\xi}(\lambda,\rho)$. For a $\Psi\in L^2(\mathbb{T})$, $\mathcal G_{\varepsilon,\xi}(\lambda,\rho)\Psi=0$ if and only if $\mathcal Q^\mathfrak{g}_{\varepsilon}(\rho,\xi)\Psi=\lambda \Psi$, where
\begin{equation}\label{e:op9}
	\mathcal{Q}^\mathfrak{g}_\varepsilon(\rho,\xi):=k^2(\partial_z+i\xi)(c-\J_k-2\alpha_1\eta_\mathfrak{g}-3\alpha_2\eta^2_\mathfrak{g})+(\g+\rho^2)(\partial_ z+i\xi)^{-1}.
\end{equation}
Therefore, the operator $\mathcal G^\mathfrak{g}_{\varepsilon,\xi}(\lambda,\rho)$ is not  invertible in $L^2(\mathbb{T})$ for some $\lambda\in \C$ and $\xi\neq 0$ if and only if  $\lambda\in\operatorname{spec}_{L^2(\mathbb{T})}(\mathcal Q^\mathfrak{g}_{\varepsilon}(\rho,\xi))$, $L^2(\mathbb{T})$-spectrum of the operator. This is straight forward to observe that $\operatorname{spec}_{L^2(\mathbb{T})}(\mathcal Q^\mathfrak{g}_{\varepsilon}(\rho,\xi))$ is not symmetric with respect to the reflections through the real axis and the origin, rather it exhibit following properties.
\begin{proposition}\label{p:2}
The operator $\mathcal Q^\mathfrak{g}_{\varepsilon}(\rho,\xi)$ possess following properties.
\begin{enumerate}
    \item The operator $\mathcal Q^\mathfrak{g}_{\varepsilon}(\rho,-\xi)$ commutes with the reflection through the real axis.
    \item The operator $\mathcal Q^\mathfrak{g}_{\varepsilon}(\rho,\xi)$ anti-commutes with the reflection through the imaginary axis.
    \item The operator $\mathcal Q^\mathfrak{g}_{\varepsilon}(\rho,-\xi)$ anti-commutes with the reflection through the origin.
    \item The spectrum of $\mathcal Q^\mathfrak{g}_{\varepsilon}(\rho,\xi)$ is symmetric with respect to the reflections through the imaginary axis.
    \item The spectrum of $\mathcal Q^\mathfrak{g}_{\varepsilon}(\rho,-\xi)$ is symmetric with respect to the reflections through the real axis and the origin.
    \end{enumerate}
\end{proposition}
\begin{proof}
The proof is similar to the proof of Proposition \ref{p:1}
\end{proof}
\subsection{Characterization of unstable spectrum}
\subsubsection{\underline{Periodic perturbations}}\label{ss:g1}
A straightforward calculation reveals that 
\begin{align}\label{E:spec01}
    \mathcal Q^\mathfrak{g}_0(\rho)e^{inz} = i\Omega^\mathfrak{g}_{n,\rho}e^{inz}\quad \text{for all}\quad n \in \mathbb{Z}\setminus \{0\}
\end{align}
where
\begin{align}\label{E:omega01}
    \Omega^\mathfrak{g}_{n,\rho} = \g\left(n-\dfrac{1}{n}\right)+k^2n(\jmath(k)-\jmath(kn))-\dfrac{\rho^2}{n}.
\end{align}
Let eigenvalues $i\Omega_{n,\rho}$ and $i\Omega_{n+\Theta,\rho}$, $n \neq m$, collide at $\rho=\rho_c>0$. From \eqref{eq:krein}, Krein signatures $\chi_n$ and $\chi_{n+\Theta}$ are opposite at $\rho=\rho_c$ when $n(n+\Theta) < 0$, i.e. $n$ and $n+\Theta$ should be of opposite parity. 
\begin{lemma}
For a fixed $\g>0$ and each $\Theta\in\N$, the potentially unstable collisions between the eigenvalues $i\Omega^\mathfrak{g}_{n,\rho}$ and $i\Omega^\mathfrak{g}_{n+\Theta,\rho}$ are
\begin{enumerate}
    \item $\rho=0$, $\{n,n+\Theta\}=\{-1,1\}$ for all $\jmath$.
    \item $\rho\neq0$, $n\in(-\Theta,0)$, $\Theta\geq3$ for all $\jmath$.
\end{enumerate}
All these collisions occur for $k$ satisfying the relation $-\g\mathscr{J}_\mathfrak{g}(n,\Theta)+\dfrac{k^2}{\Theta} \mathscr{U}_\mathfrak{g}(n,\Theta)>0$, where $\mathscr{J}_\mathfrak{g}(n,\Theta)=n(n+\Theta)+1$ and $\mathscr{U}_\mathfrak{g}(n,\Theta)=n(n+\Theta)(n(\jmath(k)-\jmath(kn))-(n+\Theta)(\jmath(k)-\jmath(k(n+\Theta)))$. Also, all collisions take place away from the origin in the complex plane except when $\Theta$ is even and $n=-\Theta/2$ in which case eigenvalues $\Omega_{n,\rho}$ and $\Omega_{-n,\rho}$ collide at the origin.
\end{lemma}

\begin{proof}
A direct calculation shows that $i\Omega^\mathfrak{g}_{n,\rho}$ and $i\Omega^\mathfrak{g}_{-t,\rho}$ collide when
\begin{equation}\label{e:cc1}
    \rho^2=\g \mathscr{J}_\mathfrak{g}(n,t)+k^2 \mathscr{U}_\mathfrak{g}(n,t)
\end{equation}
where,
\[\mathscr{J}_\mathfrak{g}(n,t)=nt-1\quad\text{and}\quad \mathscr{U}_\mathfrak{g}(n,t)=\dfrac{nt}{n+t}(n(\jmath(k)-\jmath(kn))+t(\jmath(k)-\jmath(kt)))\]
When $\rho=0$, \eqref{e:cc1} holds only for $n=t=1$.
$\mathscr{J}_\mathfrak{g}(1,1)=0$ and $\mathscr{J}_\mathfrak{g}(n,t)>0$ for all $n,t\in\N\setminus\{1\}$. $\mathscr{U}_\mathfrak{g}(1,1)=0$. $\mathscr{U}_\mathfrak{g}(n,t)<0$ when $\jmath$ is monotonically increasing and $\mathscr{U}_\mathfrak{g}(n,t)>0$ when $\jmath$ is monotonically decreasing. Therefore, for every $\jmath$, there exist $\rho\neq0$ satisfying \eqref{e:cc1}. If we seek the potentially unstable collisions between the eigenvalues $i\Omega^\mathfrak{g}_{n,\rho}$ and $i\Omega^\mathfrak{g}_{n+\Theta,\rho}$ for each $\Theta\in\N$, then value of $n$ should be $n\in(-\Theta,0)$ for every $\jmath$.
Observe that $\Omega^\mathfrak{g}_{n,\rho_c}=\Omega^\mathfrak{g}_{-n,\rho_c}=0$ for $\rho_c^2=|\g(n^2-1)+k^2n^2(\jmath(k)-\jmath(kn))|$. Therefore, $\Omega^\mathfrak{g}_{n,\rho_c}$ and $\Omega^\mathfrak{g}_{n+\Theta,\rho_c}$ collide at the origin when $\Theta$ is even and $n=-\Theta/2$. All other collisions are away from origin.

\end{proof}
\subsubsection{\underline{Localized or bounded perturbations}}\label{ss:g2} A simple calculation yields the following
\begin{align}\label{e:spec12}
    \mathcal Q^\mathfrak{g}_{0}(\rho,\xi)e^{inz} = i\Omega^\mathfrak{g}_{n,\rho,\xi}e^{inz}\quad \text{for all}\quad n \in \mathbb{Z}.
\end{align}
where 
\begin{align}\label{E:omega}
    \Omega^\mathfrak{g}_{n,\rho,\xi} = \g\left(n+\xi-\dfrac{1}{n+\xi}\right)+k^2(n+\xi)(\jmath(k)-\jmath(k(n+\xi)))-\dfrac{\rho^2}{n+\xi}.
\end{align}

Therefore, the $L^2(\T)$-spectrum of $\mathcal Q_0(\rho,\xi)$ is given by
\begin{equation}\label{e:spec01}
    \operatorname{spec}_{L^2(\mathbb{T})}(\mathcal Q_0(\rho,\xi))=\{i\Omega^\mathfrak{g}_{n,\rho,\xi}; n \in \Z, \xi\in(0,1/2] \}
\end{equation}
Let eigenvalues $i\Omega_{n,\rho,\xi}$ and $i\Omega_{n+\Theta,\rho,\xi}$, $\Theta\in\N$, collide at $\rho=\rho_c>0$. From \eqref{e:krsig}, Krein signatures $\chi_n$ and $\chi_{n+\Theta}$ are opposite at $\rho=\rho_c$ when $(n+\xi)(n+\Theta+\xi) < 0$.
\begin{lemma}
For a fixed $\g>0$ and each $\Theta\in\N$, the potentially unstable collisions between the eigenvalues $i\Omega^\mathfrak{g}_{n,\rho,\xi}$ and $i\Omega^\mathfrak{g}_{n+\Theta,\rho,\xi}$ are all $n\in[-\Theta,-1]$ for each $\Theta\in\N\setminus\{2\}$ when $\jmath$ is increasing, and all $n\in[-\Theta,-1]$ for each $\Theta\in\N\setminus\{1\}$ when $\jmath$ is decreasing. All these collisions occur for $k$ satisfying the relation $-\g \mathscr{F}_\mathfrak{g}(n,\xi,\Theta)+ k^2\mathscr{V}_\mathfrak{g}(n,\xi,\Theta)>0$, where $\mathscr{F}_\mathfrak{g}(n,\xi,\Theta)=(n+\xi)(n+\xi+\Theta)+1 $ and $\mathscr{V}_\mathfrak{g}(n,\xi,\Theta)=\dfrac{(n+\xi)(n+\xi+\Theta)}{\Theta}((n+\xi)(\jmath(k)-\jmath(k(n+\xi)))-(n+\xi+\Theta)(\jmath(k)-\jmath(k(n+\xi+\Theta))))$. All collisions take place away from the origin in the complex plane except when $\Theta$ is odd and $n=-(\Theta+1)/2$ in which case eigenvalues $\Omega_{n,\rho,\Theta}$ and $\Omega_{-n-1,\rho,\Theta}$ collide at the origin for $\rho=\rho_c(1/2)$. Table~\ref{tab:col_nonp} provides range of values of $\xi$ for all such collisions.
\begin{table}[ht]
    \centering
    \begin{tabular}{|c|c|c|c|c|}
    \hline
  $\jmath$ & $\Theta$ & $n$ & $\xi$ \\ \hline
        increasing & $\mathbb{N}\setminus\{2\}$ &$(-\Theta,-1]$ & $(0,1/2]$\\ \hline
        increasing & $1$ &$-\Theta$  & $(0,1/2]$\\ \hline
        increasing & $\geq 3$ &$-\Theta$ & $\left((\Theta-\sqrt{\Theta^2-4})/2,1/2\right] $\\ \hline
        decreasing  & $\N\setminus\{1\}$ & $[-\Theta,-1]$ & $(0,1/2]$\\ \hline
    \end{tabular}
\caption{For a given nature of $\jmath$ and value(s) of $\Theta$, each row lists value(s) of $n$ for which collisions takes place along with the value(s) of $\xi$. }
    \label{tab:col_nonp1}
\end{table}
\end{lemma}
\begin{proof}
The collision condition for eigenvalues $i\Omega^\mathfrak{g}_{n,\rho,\xi}$ and $i\Omega^\mathfrak{g}_{n+\Theta,\rho,\xi}$ for $\Theta\in\N$ is
\begin{equation}\label{e:gcc11}
    \rho^2=-\g \mathscr{F}_\mathfrak{g}(n,\xi,\Theta)+ k^2\mathscr{V}_\mathfrak{g}(n+\Theta,\xi,\Theta)
\end{equation}
where,
\[\mathscr{F}_\mathfrak{g}(n,\xi,\Theta)=(n+\xi)(n+\xi+\Theta)+1\quad\text{and}\quad \]\[\mathscr{V}_\mathfrak{g}(n,\xi,\Theta)=\dfrac{(n+\xi)(n+\xi+\Theta)}{\Theta}((n+\xi)(\jmath(k)-\jmath(k(n+\xi)))-(n+\xi+\Theta)(\jmath(k)-\jmath(k(n+\xi+\Theta))))\]
Since we intend to find only potentially unstable eigenvalues, therefore $(n+\xi)(n+\xi+\Theta)<0$ which implies $n\in[-\Theta,-1]$. 

   The function $\mathscr{F}_\mathfrak{g}(n,\xi,1),\mathscr{F}_\mathfrak{g}(n,\xi,2)>0$ for all $n\in[-\Theta,-1]$; while $\mathscr{F}_\mathfrak{g}(n,\xi,\geq3)$ is positive only for $n=-\Theta$ for all $\xi\in\left(0,\dfrac{\Theta-\sqrt{\Theta^2-4}}{2}\right)$; and negative for all
\begin{enumerate}
    \item $n\in[-\Theta+1,-1]$ for all $\Theta\in(0,1/2]$
    \item $n=-\Theta$ for all $\xi\in\left(\dfrac{\Theta-\sqrt{\Theta^2-4}}{2},\dfrac{1}{2}\right]$
\end{enumerate}
Also The sign of the function $\mathscr{V}_\mathfrak{g}(n,\xi,\Theta)$ is
\begin{enumerate}
\item When $\jmath$ is monotonically increasing
\begin{enumerate}
    \item The function $\mathscr{V}_\mathfrak{g}(n,\xi,1)>0$ for all $n\in[-\Theta,-1]$ and $\xi\in(0,1/2]$
    \item The function $\mathscr{V}_\mathfrak{g}(n,\xi,\geq2)<0$ for all $n\in[-\Theta,-1]$ and $\xi\in(0,1/2]$.
\end{enumerate}
\item When $\jmath$ is monotonically decreasing
\begin{enumerate}
    \item The function $\mathscr{V}_\mathfrak{g}(n,\xi,1)<0$ for all $n\in[-\Theta,-1]$ and $\xi\in(0,1/2]$
    \item The function $\mathscr{V}_\mathfrak{g}(n,\xi,\geq2)>0$ for all $n\in[-\Theta,-1]$ and $\xi\in(0,1/2]$. 
\end{enumerate}
\end{enumerate}
The proof follows trivially from here.
Note that $\Omega^\mathfrak{g}_{n,\rho,\xi}=0$ at $\rho^2=|\g((n+\xi)^2-1)+k^2(n+\xi)^2(\jmath(k)-\jmath(k(n+\xi)))|$. $\Omega^\mathfrak{g}_{n,\rho_c,\xi}=\Omega^\mathfrak{g}_{n+\Theta,\rho_c,\xi}$=0 for a fixed $\rho_c$ is possible only for $\Theta=-2n-1$, $\xi=1/2$. Therefore, $\Omega^\mathfrak{g}_{n,\rho_c,\xi}$ and $\Omega^\mathfrak{g}_{n+\Theta,\rho_c,\xi}$ collide at the origin for $n=-(\Theta+1)/2$, for all $n\in[-\Theta,-1]\cap \mathbb{Z}$, $\xi=1/2$ and $\rho_c^2=|\g((n+1/2)^2-1)+k^2(n+1/2)^2(\jmath(k)-\jmath(k(n+1/2)))|$. All other collisions are away from origin.

\end{proof}
\subsection{Long wavelength transverse perturbations}
We start the further analysis with the values of $\rho$ sufficiently close to origin, that is, $|\rho|\leq\rho_0
$ for some $\rho_0>0$. Table \ref{tab:col} work as it is here too.  
For $\varepsilon$ and $\rho$ sufficiently small, $\{\Omega^\mathfrak{g}_{-1,\rho},\Omega^\mathfrak{g}_{1,\rho}\}$ are pair of eigenvalues bifurcating continuously from $\{\Omega^\mathfrak{g}_{-1,0},\Omega^\mathfrak{g}_{1,0}\}$. For $\varepsilon=0$, $\{\Omega^\mathfrak{g}_{-1,0},\Omega^\mathfrak{g}_{1,0}\}$ are equipped with eigenfunctions $\{e^{-iz},e^{iz}\}$. We choose the real basis $\{\cos z,\sin z\}.$ We calculate expansion of a basis $\{\psi_1,\psi_2\}$ for the eigenspace corresponding to the eigenvalues of $\{\Omega^\mathfrak{g}_{-1,\rho},\Omega^\mathfrak{g}_{1,\rho}\}$ in $L^2_0(\mathbb{T})$ by using expansions of $\eta_\mathfrak{g}$ and $c_\mathfrak{g}$ in \eqref{E:w_ansatz}, for small $\varepsilon$ and $\rho$ as 
\begin{align*}
	\psi^\mathfrak{g}_{1}(z)  & =\cos z+2 \varepsilon A_{2} \cos 2 z+O(\varepsilon^{2}),\\
	\psi^\mathfrak{g}_{2}(z) & =\sin z+2 \varepsilon A_{2} \sin 2 z+O(\varepsilon^{2}).
\end{align*} 
We have the following expression for  $\mathcal{Q}^\mathfrak{g}_{\varepsilon}(\rho)$ after expanding and using $\eta$ and $c$
\begin{align*}
    \mathcal{Q}^\mathfrak{g}_{\varepsilon}(\rho)= \g(\partial_z+\partial_z^{-1})+\rho^2\partial_z^{-1}+k^2\partial_z(\jmath(k)-\J_k)&-2\varepsilon \alpha_1k^2\partial_z(\cos z)+k^2\varepsilon^2(\alpha_1\eta_{\mathfrak{g}_2}-3/4\alpha_2)\partial_z\\&-k^2\varepsilon^2(2\alpha_1\eta_{\mathfrak{g}_2}+3/2\alpha_2)\partial_z(\cos 2z)+O(\varepsilon^3)
\end{align*}
We use expansion of $\mathcal Q^\mathfrak{g}_\varepsilon(\rho)$ to find actions of $\mathcal{Q}^\mathfrak{g}_\varepsilon(\rho)$ and identity operator on $\{\psi^\mathfrak{g}_1,\psi^\mathfrak{g}_2\}$, and arrive at
\begin{align*}
	\mathcal{T}^\mathfrak{g}_\varepsilon(\rho)&=\begin{pmatrix}
	0 & \rho^2+\dfrac{3}{2}\alpha_2 k^2\varepsilon^2+2\varepsilon^2\alpha_1k^2\eta_{\mathfrak{g}_2} \\ -\rho^2 & 0\end{pmatrix}+O(\varepsilon^2(\rho+\varepsilon)),\\
	\end{align*}
	To locate where these two eigenvalues are bifurcating from the origin, we analyze the characteristic equation $	\left|\mathcal{T}^\mathfrak{g}_\varepsilon(\rho)-\lambda_\mathfrak{g} \mathcal{I}_\varepsilon\right|=0$, where $\mathcal{I}_\varepsilon$ is $2\times 2$ identity matrix. From which we conclude that
	\begin{equation}
	    \lambda_\mathfrak{g}^2+\rho^2(\rho^2+\dfrac{3}{2}\alpha_2 k^2\varepsilon^2+2\varepsilon^2\alpha_1k^2\eta_{\mathfrak{g}_2})=O(|\varepsilon|\rho^2(\rho^2+\varepsilon^2)
	\end{equation}
	From which we conclude that
	\begin{equation}
	    \lambda_\mathfrak{g}^2=-\rho^2(\rho^2-\dfrac{3}{2}|\alpha_2| k^2\varepsilon^2+2\varepsilon^2|\alpha_1|k^2\eta_{\mathfrak{g}_2})+O(|\varepsilon|\rho^2(\rho^2+\varepsilon^2)
	\end{equation}
	For $\rho=\varepsilon=0$, we get zero as a double eigenvalue, which agrees with our calculation.
	\begin{enumerate}
    \item $\alpha_1=1$, $\alpha_2=0$
    \begin{enumerate}
        \item $\jmath$ is increasing, then we obtain two real eigenvalues with opposite signs when \[|k^2(\jmath(2k)-\jmath(k))|>\dfrac{3\g}{4}\quad\text{and}\quad\rho^2<2\varepsilon^2k^2|\eta_{\mathfrak{g}_2}|\]
        \item $\jmath$ is decreasing, we obtain two purely imaginary eigenvalues.
    \end{enumerate}
    \item $\alpha_1=0$, $\alpha_2=-1$; we obtain two real eigenvalues with opposite signs for  all $\jmath$ and $\forall k>0$ and
    \[|\rho|<\sqrt{\dfrac{3}{2}}|\varepsilon|k\].
    \item $\alpha_1=1$, $\alpha_2=-1$, we obtain two real eigenvalues with opposite signs for  all $\jmath$ when
    \[\dfrac{2k^2}{3\g+4k^2(\jmath(k)-\jmath(2k))}<\dfrac{3}{4}\quad\text{and}\quad\rho^2<\dfrac{3}{2}|\alpha_2| k^2\varepsilon^2-2\varepsilon^2|\alpha_1|k^2|\eta_{\mathfrak{g}_2}|\]
    \end{enumerate}

 Hence the Theorem \ref{t:5}.
 
\subsection{Finite or short transverse perturbations}
Here we work with the values of $\rho$ away from the origin, $|\rho|\geq\rho_0$ for some $\rho_0>0$. We do further analysis to check if collisions in Table \ref{tab:col} indeed lead to instability or not.
\subsubsection{\underline{(In)stability analysis for $\Theta=1$}:}
For periodic perturbations, there are no collisions mentioned for $\Theta=1$. However, there in one pair of colliding eigenvalues $\{\Omega^\mathfrak{g}_{-1,\rho,\xi}, \Omega^\mathfrak{g}_{0,\rho,\xi}\}$ for $\b>0$ with respect to non-periodic perturbations, colliding for
\begin{equation}\label{e:rc1}
    \rho^2=-\g(1-\xi+\xi^2)+ k^2(\xi(\xi-1)(\jmath(k)-\jmath(k(\xi-1)))-\xi(\jmath(k)-\jmath(k\xi))
\end{equation}
There exists a curve $\rho=\rho_c$ given in \eqref{e:rc1} along which \[\Omega^\mathfrak{g}(\rho_c,\xi):=\Omega^\mathfrak{g}_{-1,\rho_c,\xi} =\Omega^\mathfrak{g}_{0,\rho_c,\xi}.\]
Furthermore,
\begin{align}
 \phi^\mathfrak{g}_{0,-1}(z) = e^{-i z}
 \quad \quad and \quad\quad
 \phi^\mathfrak{g}_{0,0}(z) = 1
\end{align}
forms the corresponding eigenspace for $\mathcal{Q}^\mathfrak{g}_0(\rho_c,\xi)$ associated with the two eigenvalues. Let
\begin{align}
    i \Omega^\mathfrak{g}(\rho_c,\xi) + i \theta^\mathfrak{g}_{\varepsilon,\rho,-1}
    \quad \quad and \quad\quad
    i \Omega^\mathfrak{g}(\rho_c,\xi) + i \theta^\mathfrak{g}_{\varepsilon,\rho,0}
\end{align}
are the eigenvalues of $\mathcal{Q}^\mathfrak{g}_\varepsilon(\rho_c,\xi)$ departing from $i\Omega^\mathfrak{g}_{-1,\rho_c,\xi}$ and $i\Omega^\mathfrak{g}_{0,\rho_c,\xi}$ respectively for sufficiently small $|\varepsilon|$ and $|\rho-\rho_c|$. Let $\{\phi^\mathfrak{g}_{\varepsilon,\rho,-1}(z), \phi^\mathfrak{g}_{\varepsilon,\rho,0}(z)\}$ be the extended eigenspace associated with two bifurcating eigenvalues. Following \cite{Creedon2021High-FrequencyApproach}, we can take,
\begin{align}\label{eq:eigg1}
    \phi^\mathfrak{g}_{\varepsilon,-1,\rho}(z) =& e^{-iz}+\varepsilon \phi^\mathfrak{g}_{-1}+O(\varepsilon^2), \\
    \phi^\mathfrak{g}_{\varepsilon,0,\rho}(z) =& 1+\varepsilon\phi^\mathfrak{g}_{0}+O(\varepsilon^2)\label{eq:eigg2}
\end{align}
Using the orthonormality conditions on the eigenfunctions $\phi^\mathfrak{g}_{\varepsilon,\rho,-1}$ and $\phi^\mathfrak{g}_{\varepsilon,\rho,0}$, we get
\[
\phi^\mathfrak{g}_{-1}=\phi^\mathfrak{g}_{0} = 0.
\]
In order to find the eigenvalues, we compute matrix representations of $\mathcal{Q}^\mathfrak{g}_\varepsilon(\rho_c,\xi)$ and identity operators on $\{\phi^\mathfrak{g}_{\varepsilon,-1,\rho}(z), \phi^\mathfrak{g}_{\varepsilon,0,\rho}(z)\}$ for sufficiently small $|\varepsilon|$ and $|\rho-\rho_c|$ 
\[
\mathcal{B}^\mathfrak{g}_\varepsilon(\rho,\xi) = 
\begin{pmatrix}i\Omega^\mathfrak{g}(\rho_c,\xi)-i\dfrac{\varsigma}{\xi-1}&
-i \varepsilon\alpha_1 k^2\xi
\\
-i\varepsilon \alpha_1k^2(\xi-1) &i\Omega^\mathfrak{g}(\rho_c,\xi)-i\dfrac{\varsigma}{\xi}
\end{pmatrix}+O(|\varepsilon|(|\varsigma|+|\varepsilon|)),
\]
where $\varsigma=\rho^2-\rho_c^2$ and
$\mathcal{I}_\varepsilon$ is the $2\times 2$ identity matrix. We solve the characteristic equation $\det(\mathcal{B}^\mathfrak{g}_\varepsilon(\rho,\xi)-\lambda^\mathfrak{g}\mathcal{I}_\varepsilon)=0$ for $\lambda$ of the form
\[
\lambda^\mathfrak{g} = i\Omega^\mathfrak{g}(\rho_c,\xi) +i\theta,
\]
and obtain the polynomial equation
\[
\theta^2+\theta\varsigma\left(\dfrac{1}{\xi-1} + \frac{1}{\xi} +
O(\varepsilon^2)\right)
-\varepsilon^2\alpha_1^2k^4\xi(\xi-1)+\frac{\varsigma^2}{\xi(\xi-1)} +
  O(\varepsilon^2(|\varsigma|+\varepsilon^2)) =0.
\]
A direct calculation depicts that the discriminant of this polynomial is
\[
\operatorname{disc}_\varepsilon(\varsigma,\xi) = 
\frac{\varsigma^2}{\xi^2(\xi-1)^2} +4\alpha_1^2k^4\xi(\xi-1)\varepsilon^2+
O(\varepsilon^2(|\varsigma|+\varepsilon^2)).
\]For any $\varepsilon$ sufficiently small there exists
\[
\varsigma_\varepsilon(\xi) = 2\alpha_1k^2\xi^{3/2}(1-\xi)^{3/2}|\varepsilon|+O(\varepsilon^2)>0
\]
such that the two eigenvalues of $\mathcal Q^\mathfrak{g}_\varepsilon(\rho,\xi)$  are purely imaginary when $|\rho^2-\rho_c^2|\geq \varsigma_\varepsilon(\xi)$ and complex with opposite nonzero real parts when $|\rho^2-\rho_c^2|<\varsigma_\varepsilon(\xi)$, which proves the theorem \ref{t:6}.

\subsubsection{\underline{(In)stability analysis for $\Theta\geq2$}:}
For some $n \in \mathbb{Z}$ and a fixed $\Theta\geq2$, we have
\begin{equation}\label{e:col_cond}
 i \Omega^\mathfrak{g}_{n, \rho_\ast,\xi}=i \Omega^\mathfrak{g}_{n+\Theta, \rho_\ast,\xi}=i \Omega(\rho_\ast,\xi), \quad \xi\in (-1/2,1/2]
\end{equation}
$\xi=0$ represents the periodic case and $\xi\neq0$ represents the non-periodic case. $i\Omega^\mathfrak{g}$ is an eigenvalue of $\mathcal{Q}^\mathfrak{g}_0(\rho_c,\xi)$ of multiplicity two with an orthonormal basis of eigenfunctions $\{e^{inz},e^{i(n+\Theta)z}\}$. Let $i \Omega^\mathfrak{g} + i \nu_{a,n}$ and $i \Omega^\mathfrak{g} + i \nu_{a,n+2}$ be the eigenvalues of $\mathcal{Q}^\mathfrak{g}_\varepsilon(\rho,\xi)$ bifurcating from $i\Omega^\mathfrak{g}_{n,\rho_c,\xi}$ and $i\Omega^\mathfrak{g}_{n+2,\rho_c,\xi}$ respectively, for $|\varepsilon|$ and $|\rho-\rho_c|$ small. For the corresponding eigenspace, suppose $\{\varphi^\mathfrak{g}_{\varepsilon,n}(z), \varphi^\mathfrak{g}_{\varepsilon,n+2}(z)\}$ be an orthonormal basis. We suppose the following expansions 
\begin{align}\label{eq:eiggg1}
    \varphi^\mathfrak{g}_{\varepsilon,n}(z) =& e^{inz}+\varepsilon\varphi_{n,1}+\varepsilon^2\varphi_{n,2}+\dots+\varepsilon^\Theta\varphi_{n,\Theta}+O(\varepsilon^{\Theta+1}), \\
    \varphi^\mathfrak{g}_{\varepsilon,n+\Theta}(z) =& e^{i(n+\Theta)z}+\varepsilon\varphi_{n+\Theta,1}+\varepsilon^2\varphi_{n+\Theta,2}+\dots+\varepsilon^\Theta\varphi_{n+\Theta,\Theta}+O(\varepsilon^{\Theta+1}).\label{eq:eiggg2}
\end{align}
We use orthonormality of $\varphi^\mathfrak{g}_{\varepsilon,n}$ and $\varphi^\mathfrak{g}_{\varepsilon,n+\Theta}$ to find that
\[
\varphi^\mathfrak{g}_{n,1} = \varphi^\mathfrak{g}_{n,2} =\dots=\varphi^\mathfrak{g}_{n,\Theta}= \varphi^\mathfrak{g}_{n+\Theta,1} = \varphi^\mathfrak{g}_{n+\Theta,2} = \dots = \varphi^\mathfrak{g}_{n+\Theta,\Theta}=0.\]
Using the expansions of $\eta$ and $c$ in \eqref{e:expptw}, we expand $\mathcal{Q}^\mathfrak{g}_\varepsilon(\rho,\xi)$ in $\varepsilon$ as
\begin{equation}
\begin{aligned}\label{e:op0111}
    \mathcal{Q}^\mathfrak{g}_\varepsilon(\rho,\xi)=\mathcal{Q}_0(\rho,\xi)+k^2(\b_2\varepsilon^2+\b_4\varepsilon^4+\dots)(\partial_z+i\xi)&-2k^2\varepsilon\delta_1(\partial_z+i\xi)(\cos z)\\&-\dots-2k^2\varepsilon^\Theta\delta_\Theta(\partial_z+i\xi)(\cos (\Theta z))
\end{aligned}
\end{equation}

Following the same procedures as in the preceding subsection, we arrive at
\begin{align*}
	\mathscr{G}^\mathfrak{g}_\varepsilon(\rho,\xi)&=\begin{pmatrix}
		\mathscr{G}^\mathfrak{g}_{11} & \ \mathscr{G}^\mathfrak{g}_{12} \\ \mathscr{G}^\mathfrak{g}_{21} & \mathscr{G}^\mathfrak{g}_{22}\end{pmatrix}+O(\varepsilon^{\Theta+1}),\\
	\end{align*}
where 
\begin{align*}
		\mathscr{G}^\mathfrak{g}_{11}& =i\Omega(\rho_\ast,\xi)-i\dfrac{\varsigma}{n+\xi}+ik^2(n+\xi)(\b_2\varepsilon^2+\b_4\varepsilon^4+\dots)  , \\
	\mathscr{G}^\mathfrak{g}_{12} & =-ik^2(n+\xi+\Theta)\delta^\Theta\varepsilon^\Theta,\\
	\mathscr{G}^\mathfrak{g}_{21}&=-ik^2(n+\xi)\delta^\Theta\varepsilon^\Theta,\\ \mathscr{G}^\mathfrak{g}_{22}&=i\Omega(\rho_\ast,\xi)-i\dfrac{\varsigma}{n+\xi+\Theta}+ik^2(n+\xi+\Theta)(\b_2\varepsilon^2+\b_4\varepsilon^4+\dots),
\end{align*} 
The resulting discriminant of the characteristic equation $\det(\mathscr{G}^\mathfrak{g}_{\varepsilon}(\rho,\xi)-(i \Omega(\rho_\ast,\xi) + i \theta) \mathcal{I}_{\varepsilon}) = 0$ is
\begin{align*}
    \operatorname{disc}^\mathfrak{g}_\varepsilon(\varsigma) = \dfrac{\Theta^2\varsigma^2}{(n+\xi)^2(n+\xi+\Theta)^2}+k^4\Theta^2\beta_2^2\varepsilon^4+O(\varepsilon^2(|\varsigma|+|\varepsilon^3|)).
\end{align*}
For sufficiently small $|\varsigma|$ and $|\varepsilon|$, $\operatorname{disc}^\mathfrak{g}_\varepsilon(\varsigma)$ is positive which implies that no eigenvalue of $\mathcal{Q}^\mathfrak{g}_\varepsilon(\rho,\xi)$ is bifurcating from the imaginary axis due to collision. Hence the proof of Theorems \ref{t:5} and \ref{t:6}.

\section{Applications}\label{sec:app}
\subsection{RM-fKdV-KP Equation}\label{ss:1}
The RM-fKdV-KP equation can be derived from \eqref{e:gRMKP} by choosing
\[
\jmath(k)= 1+|k|^\alpha, \qquad \alpha>1/2
\]  with $\alpha_1=1$, $\alpha_2=0$. 
Clearly, the symbol $\jmath(k)$ meets the hypotheses~\ref{h:m}~$J1$, $J2$ ($\mathfrak{b}=\alpha$,$A_1=1$ and $A_2=2$), and $J3$ ($\jmath$ is strictly increasing for $k>0$). We can obtain the periodic solutions of rotation-modified fKdV equation from \eqref{E:w_ansatz} by replacing $\jmath(k)$ with $1+|k|^\alpha$. Transverse stability and instability of these solutions can be derived using Theorems~\ref{t:5}, \ref{t:6}, \ref{t:7} and \ref{t:8}.
\begin{corollary}[Transverse stability vs. instability of RM-fKdV-KP]\label{c:RMKPfkdv}
For any $\varepsilon$ sufficiently small and $\g>0$, 
\begin{enumerate}
\item 
\begin{enumerate}
    \item For all $k>0$ and $\b<0$, periodic traveling waves \eqref{E:w_ansatz} of \eqref{e:RMFKP} are transversely stable with respect to either periodic or non-periodic (localized or bounded) perturbations in the direction of propagation and periodic perturbations in the transverse direction.
    \item For all $k>0$ and $\b>0$, periodic traveling waves \eqref{E:w_ansatz} of \eqref{e:RMFKP} are transversely stable with respect to periodic perturbations in the direction of propagation and, periodic with finite wavelength perturbations in the transverse direction.
\end{enumerate}
\item  \begin{enumerate}\item $\b>0$, periodic traveling waves \eqref{E:w_ansatz} of \eqref{e:RMFKP} are transversely unstable with  respect to periodic perturbations in both directions and long wavelength perturbations in the transverse direction if 
\[
k>\left|\dfrac{3\g}{4\b(2^\alpha-1)}\right|^{1/\alpha+2}\]
\item $\b>0$, periodic traveling waves \eqref{E:w_ansatz} of \eqref{e:RMFKP} are transversely unstable with  respect to non-periodic (localized or bounded) perturbations in the direction of propagation of the wave and, periodic with finite wavelength perturbations in the transverse direction if 
\[
k>\left(\dfrac{3(2^\alpha)\g}{(2^\alpha-1)\b }\right)^{1/\alpha+2}\]
\end{enumerate}

\end{enumerate}
\end{corollary}
Note that, for $\g=0$, all the results mentioned in Corollary \ref{c:RMKPfkdv} agree with the findings in \cite{Bhavna2022TransverseEquation} for KP-fKdV equation.
\subsection{RMBO-KP Equation}\label{ss:2}
RMBO-KP equation analogous to RM-fKdV-KP equation with $\alpha=1$. We have derived transverse stability and instability of these solutions in Corollary \ref{c:RMKPBO} using Theorems~\ref{t:5}, \ref{t:6}, \ref{t:7} and \ref{t:8}.
\begin{corollary}[Transverse stability vs. instability of RMKP-BO]\label{c:RMKPBO}
For any $\varepsilon$ sufficiently small and $\g>0$, 
\begin{enumerate}
\item 
\begin{enumerate}
    \item For all $k>0$ and $\b<0$, periodic traveling waves \eqref{E:w_ansatz} of \eqref{e:RMBOKP} are transversely stable with respect to either periodic or non-periodic (localized or bounded) perturbations  in the direction of propagation and periodic perturbations  in the transverse direction.
    \item For all $k>0$ and $\b>0$, periodic traveling waves \eqref{E:w_ansatz} of \eqref{e:RMBOKP} are transversely stable with respect to periodic perturbations in the direction of propagation and, periodic with finite wavelength perturbations in the transverse direction.
\end{enumerate}
\item  \begin{enumerate}\item $\b>0$, periodic traveling waves \eqref{E:w_ansatz} of \eqref{e:RMBOKP} are transversely unstable with  respect to periodic perturbations in both directions and long wavelength perturbations in the transverse direction if 
\[
k>\left|\dfrac{3\g}{4\b}\right|^{1/3}\]
\item $\b>0$, periodic traveling waves \eqref{E:w_ansatz} of \eqref{e:RMBOKP} are transversely unstable with  respect to non-periodic (localized or bounded) perturbations  in the direction of propagation of the wave and, periodic with finite wavelength perturbations  in the transverse direction if 
\[
k>\left(\dfrac{6\g}{\b }\right)^{1/3}\]
\end{enumerate}

\end{enumerate}
\end{corollary}
For $\g=0$, all the results mentioned in Corollary \ref{c:RMKPBO} agree with the findings in \cite{Bhavna2022TransverseEquation} for KP-BO equation.
\subsection{RMG-KP Equation}\label{ss:3}
The RMG-KP equation can be derived from \eqref{e:gRMKP} by choosing
\[
\jmath(k)= 1+|k|^2, \qquad 
\]  with $\alpha_1=1$, $\alpha_2=-1$. 
Clearly, the symbol $\jmath(k)$ meets the hypotheses~\ref{h:m}~$J1$, $J2$ ($A_1=1$, and $A_2=2$), and $J3$ ($\jmath$ is strictly increasing for $k>0$). We can obtain the periodic solutions of rotation-modified Gardner equation from \eqref{E:w_ansatz} by replacing $\jmath(k)$ with $1+|k|^2$. We have derived transverse stability and instability of these solutions in Corollary \ref{c:RMGKP} using Theorems~\ref{t:5}, \ref{t:6}, \ref{t:7} and \ref{t:8}.
\begin{corollary}[Transverse stability vs. instability of RMG-KP]\label{c:RMGKP}
For any $\varepsilon$ sufficiently small and $\g>0$, 
\begin{enumerate}
\item 
\begin{enumerate}
    \item For all $k>0$ and $\b<0$, periodic traveling waves \eqref{E:w_ansatz} of \eqref{e:GKP} are transversely stable with respect to non-periodic (localized or bounded) perturbations in the direction of propagation and periodic perturbations in the transverse direction.
    \item For all $k>0$ and $\b\neq0$, periodic traveling waves \eqref{e:expptw} of \eqref{e:GKP} are transversely stable with respect to periodic perturbations in the direction of propagation and, periodic with finite wavelength perturbations in the transverse direction.
\end{enumerate}
\item \begin{enumerate} \item periodic traveling waves \eqref{E:w_ansatz} of \eqref{e:GKP} are transversely unstable with respect to periodic perturbations in both directions and long wavelength perturbations in the transverse direction if \begin{enumerate}
\item $\b>0$,
\[
36|\b|k^4+8k^2<9\g\]
\item $\b<0$,
\[
-36|\b|k^4+8k^2<9\g\]

\end{enumerate}
\item $\b>0$, periodic traveling waves \eqref{E:w_ansatz} of \eqref{e:GKP} are transversely unstable with  respect to non-periodic (localized or bounded) perturbations in the direction of propagation of the wave and, periodic with finite wavelength perturbations in the transverse direction if 
\[
k>\left(\dfrac{4\g}{\b }\right)^{1/4}\] 

\end{enumerate}

\end{enumerate}
\end{corollary}
\subsection{Reduced RMKP Equation}\label{ss:4}
The Reduced RMKP equation can be derived from \eqref{e:gRMKP} by choosing
\[
\jmath(k)= 1+|k|^2, \qquad 
\]  with $\b=0$, $\alpha_1=1$ and $\alpha_2=0$. 
Clearly, the symbol $\jmath(k)$ meets the hypotheses~\ref{h:m}~$J1$, $J2$ ($A_1=1$, and $A_2=2$), and $J3$ ($\jmath$ is strictly increasing for $k>0$). We can obtain the periodic solutions of reduced Ostrovsky equation from \eqref{E:w_ansatz} by replacing $\jmath(k)$ with $1+|k|^2$. We can obtain transverse stability and instability of these solutions using Theorems~\ref{t:5}, \ref{t:6}, \ref{t:7} and \ref{t:8}.
\begin{corollary}\label{c:4}
Assume that small-amplitude periodic traveling waves of the reduced Ostrovsky equation are spectrally stable in $L^2(\mathbb T)$ with respect to one-dimensional perturbations. Then
for any $\varepsilon$ sufficiently small, $\g>0$, $k>0$ and $\b=0$, periodic traveling waves \eqref{E:w_ansatz} of \eqref{e:RMKP} are transversely stable with respect to either periodic or non-periodic (localized or bounded) perturbations in the direction of propagation and periodic perturbations in the transverse direction.
\end{corollary}
\subsection{RM-mKdV-KP Equation}\label{ss:5}
The RM-mKdV-KP equation can be derived from \eqref{e:gRMKP} by choosing
\[
\jmath(k)= 1+|k|^2, \qquad 
\]  with $\alpha_1=0$ and $\alpha_2=-1$. 
Clearly, the symbol $\jmath(k)$ meets the hypotheses~\ref{h:m}~$J1$, $J2$ ($A_1=1$, and $A_2=2$), and $J3$ ($\jmath$ is strictly increasing for $k>0$). We can obtain the periodic solutions of rotation-modified mKdV equation from \eqref{E:w_ansatz} by replacing $\jmath(k)$ with $1+|k|^2$. We can obtain transverse stability and instability of these solutions using Theorems~\ref{t:5}, \ref{t:6}, \ref{t:7} and \ref{t:8}.
\begin{corollary}\label{c:5}
For any $\varepsilon$ sufficiently small and $\g>0$,
\begin{enumerate}
\item 
\begin{enumerate}
    \item For all $k>0$ and $\b<0$, periodic traveling waves \eqref{E:w_ansatz} of \eqref{e:mKdVKP} are transversely stable with respect to non-periodic (localized or bounded) perturbations in the direction of propagation and periodic perturbations in the transverse direction.
    \item For all $k>0$ and $\b\neq0$, periodic traveling waves \eqref{e:expptw} of \eqref{e:mKdVKP} are transversely stable with respect to periodic perturbations in the direction of propagation and, periodic with finite wavelength perturbations in the transverse direction.
\end{enumerate}
\item \begin{enumerate} \item For all $k>0$ and $\b\neq0$, periodic traveling waves \eqref{E:w_ansatz} of \eqref{e:mKdVKP} are transversely unstable with respect to periodic perturbations in both directions and  long wavelength perturbations in the transverse direction.

\item $\b>0$, periodic traveling waves \eqref{E:w_ansatz} of \eqref{e:mKdVKP} are transversely unstable with respect to non-periodic (localized or bounded) perturbations in the direction of propagation of the wave and, periodic with finite wavelength perturbations in the transverse direction if 
\[
k>\left(\dfrac{4\g}{\b }\right)^{1/4}\] 

\end{enumerate}

\end{enumerate}
\end{corollary}
For $\g=0$, all the results mentioned in Corollary \ref{c:5} are in accordance with the findings in \cite{Johnson2010TransverseEquation,Bhavna2022KDEquation} for mKP-II equation.
\subsection{RMILW-KP Equation}\label{ss:6}
The RMILW-KP equation can be derived from \eqref{e:gRMKP} by choosing,
\[
\jmath(k)= k \coth{k}
\]
The symbol $\jmath(k)$ satisfies Hypotheses~\ref{h:m} $J1$, $J2$ ($\mathfrak{b}=2$, $A_1=1$, and $A_2=2$), and $J3$ ($\jmath$ is strictly increasing for $k>0$). 
We can obtain the periodic solutions of rotation-modified ILW equation from \eqref{E:w_ansatz} by replacing $\jmath(k)$ with $k \coth{k}$. We can obtain transverse stability and instability of these solutions using Theorems~\ref{t:5}, \ref{t:6}, \ref{t:7} and \ref{t:8}.
\begin{corollary}[Transverse stability vs. instability of RMILW-KP]\label{c:RMKPILW}
For any $\varepsilon$ sufficiently small and $\g>0$, 
\begin{enumerate}
\item 
\begin{enumerate}
    \item For all $k>0$ and $\b<0$, periodic traveling waves \eqref{E:w_ansatz} of \eqref{e:RMKPW} are transversely stable with respect to either periodic or non-periodic (localized or bounded) perturbations in the direction of propagation and periodic perturbations in the transverse direction.
    \item For all $k>0$ and $\b>0$, periodic traveling waves \eqref{E:w_ansatz} of \eqref{e:gRMKP} are transversely stable with respect to periodic perturbations in the direction of propagation and, periodic with finite wavelength perturbations in the transverse direction.
\end{enumerate}
\item  \begin{enumerate}\item $\b>0$, periodic traveling waves \eqref{E:w_ansatz} of \eqref{e:RMKPW} are transversely unstable with  respect to periodic perturbations in both directions and long wavelength perturbations in the transverse direction if 
\[
k^2(2k\coth{2k}-k\coth{k})>\dfrac{3\g}{4\b}\]
\item $\b>0$, periodic traveling waves \eqref{E:w_ansatz} of \eqref{e:RMKPW} are transversely unstable with respect to non-periodic (localized or bounded) perturbations in the direction of propagation of the wave and, periodic with finite wavelength in the transverse direction if 
\[
k^2(k\coth{k}-k/2\coth{(k/2)})>\dfrac{3\g}{\b}\]
\end{enumerate}

\end{enumerate}
\end{corollary}
Note that, for $\g=0$, all the results mentioned in Corollary \ref{c:RMKPILW} agree with the results in \cite{Bhavna2022TransverseEquation} for KP-ILW equation.
\subsection{RM-Whitham-KP Equation}\label{ss:7}
The RM-Whitham-KP equation can be derived from \eqref{e:gRMKP} by choosing,
\[
\jmath(k)= \sqrt{\frac{\tanh k}{k}} 
\]
The symbol $\jmath(k)$ satisfies Hypotheses~\ref{h:m} $J1$, $J2$ with $\mathfrak{b}= -\frac{1}{2}$, $A_1=1$, and $A_2=2$, and $J3$ as $\jmath$ is strictly decreasing for $k>0$. 
We can obtain the periodic solutions of rotation-modified Whitham equation from \eqref{E:w_ansatz} by replacing $\jmath(k)$ with $\sqrt{\frac{\tanh k}{k}}$. We can obtain transverse stability and instability of these solutions using Theorems~\ref{t:5}, \ref{t:6}, \ref{t:7} and \ref{t:8}.
\begin{corollary}[Transverse stability vs. instability of RM-Whitham-KP]\label{c:RMKPwh}
For any $\varepsilon$ sufficiently small and $\g>0$, 
\begin{enumerate}
\item 
\begin{enumerate}
    \item For all $k>0$ and $\b>0$, periodic traveling waves \eqref{E:w_ansatz} of \eqref{e:RMKPW1} are transversely stable with respect to either periodic or non-periodic (localized or bounded) perturbations in the direction of propagation and periodic perturbations in the transverse direction.
    \item For all $k>0$ and $\b<0$, periodic traveling waves \eqref{E:w_ansatz} of \eqref{e:RMKPW1} are transversely stable with respect to periodic perturbations in the direction of propagation and, periodic with finite wavelength perturbations in the transverse direction.
\end{enumerate}
\item  \begin{enumerate}\item $\b<0$, periodic traveling waves \eqref{E:w_ansatz} of \eqref{e:RMKPW1} are transversely unstable with  respect to periodic perturbations in both directions and long wavelength perturbations in the transverse direction if 
\[
k^2\left(\sqrt{\frac{\tanh k}{k}}-\sqrt{\frac{\tanh 2k}{2k}}\right)>\dfrac{3\g}{4|\b|}\]
\item  $\b<0$, periodic traveling waves \eqref{E:w_ansatz} of \eqref{e:RMKPW1} are transversely unstable with  respect to non-periodic (localized or bounded) perturbations in the direction of propagation of the wave and, periodic with finite wavelength perturbations in the transverse direction if 
\[
k^2\left(\sqrt{\frac{\tanh k/2}{k/2}}-\sqrt{\frac{\tanh k}{k}}\right)>\dfrac{3\g}{|\b|}\]
\end{enumerate}

\end{enumerate}
\end{corollary}
For $\g=0$, all the results mentioned in Corollary \ref{c:RMKPwh} agree with the findings in \cite{Bhavna2022TransverseEquation} for KP-Whitham equation.

\subsection*{Acknowledgement} 
Bhavna and AKP are supported by the Science and Engineering Research Board (SERB), Department of Science and Technology (DST), Government of India under grant 
SRG/2019/000741. Bhavna is also supported by Junior Research Fellowships (JRF) by University Grant Commission (UGC), Government of India.
AS thanks the Institute for Computational and Experimental Research in Mathematics, Providence, RI, being resident during the ``Hamiltonian Methods in Dispersive and Wave Evolution Equations" program supported by NSF-DMS-1929284.
\subsection*{Declaration of interests} The authors report no conflict of interest.
\subsection*{Data availability}
This article has no additional data.
\bibliographystyle{amsalpha}
\bibliography{RMKP.bib}
\end{document}